\documentclass[11pt,a4paper,leqno,twoside]{amsart}%
\usepackage{amsmath,amssymb,amsthm,enumerate,hyperref,color}
\numberwithin{equation}{section}
\newtheorem{theorem}{Theorem}[section]
\newtheorem{definition}[theorem]{Definition}

\newtheorem{proposition}[theorem]{Proposition}
\newtheorem{lemma}[theorem]{Lemma}
\newtheorem{corollary}[theorem]{Corollary}
\newtheorem{remark}[theorem]{Remark}

\newcommand{\rad}{{\text{\upshape rad}}}

\def\e{{\varepsilon}}

\def\d{\delta}

\def\L{{\Lambda}}
\def\l{{\lambda}}
\def\a{{\alpha}}
\def\b{{\beta}}

\newcommand{\R}{\mathbb{R}}

\newcommand{\loc}{\mathop{\mathrm{loc}}}

\newcommand{\jcal}{{\mathcal J}}
\newcommand{\n}{\mathop{N}}

\newif\ifcomment \commentfalse
\def\commentON{\commenttrue}

\long\outer\def\BC#1\EC{\ifcomment \sloppy \par \# \ldots\dotfill
{\em #1} \dotfill \# \par \fi } \commentON

\newcommand{\remove}[1]{}

\def\sideremark#1{\ifvmode\leavevmode\fi\vadjust{\vbox to0pt{\vss
 \hbox to 0pt{\hskip\hsize\hskip1em
 \vbox{\hsize2.1cm\tiny\raggedright\pretolerance10000
  \noindent #1\hfill}\hss}\vbox to15pt{\vfil}\vss}}}%

\definecolor{cadmiumgreen}{rgb}{0.0, 0.42, 0.24}
\definecolor{darkgreen}{rgb}{0.0, 0.5, 0.2}
\definecolor{purple}{rgb}{0.5, 0.0, 0.5}
\newcommand{\AL}{\color{purple}}

\title[Asymptotic profile and Morse index]{Asymptotic profile and Morse index of nodal radial solutions to the H\'enon problem}
\author[A.~L.~Amadori, F.~Gladiali]{Anna Lisa Amadori$^\dag$,  Francesca Gladiali$^\ddag$}
\thanks{This work was supported by Gruppo Nazionale per l'Analisi Matematica, la Probabilit\`a e le loro Applicazioni (GNAMPA) of the Istituto Nazionale di Alta Matematica (INdAM). {The second author is supported by Prin-2015KB9WPT}}
\date{\today}
\address{$\dag$ Dipartimento di Scienze e Tecnologie, Universit\`a di Napoli ``Parthenope", Centro Direzionale di Napoli, Isola C4, 80143 Napoli, Italy. \texttt{annalisa.amadori@uniparthenope.it}}
\address{$\ddag$ Dipartimento di Chimica e Farmacia, Universit\`a di Sassari, via Piandanna 4, 07100 Sassari, Italy. \texttt{fgladiali@uniss.it}}

\begin{document}
\maketitle 
 \begin{abstract}
 	We compute the Morse index of nodal radial solutions to the H\'enon problem
 	\[\left\{\begin{array}{ll}
 	-\Delta u = |x|^{\alpha}|u|^{p-1} u \qquad & \text{ in } B, \\
 	u= 0 & \text{ on } \partial B,
 	\end{array} \right.
 	\]
 	where $B$ stands for the unit ball in $\R^N$ in dimension $N\ge 3$, $\a>0$ and $p$ is near at the threshold exponent for existence of solutions $p_{\a}=\frac{N+2+2\a}{N-2}$, obtaining that
 	\begin{align*} 
 	m(u_p) & = m \sum\limits_{j=0}^{1+\left[{\alpha}/{2}\right]} N_j  \quad & \mbox{ if $\a$ is not an even integer, or} \\
 	m(u_p)& =  m\sum\limits_{j=0}^{ \a /2} N_j + (m-1) N_{1+\a/ 2}  & \mbox{ 	if $\alpha$  is an even number.}
 	\end{align*}	
	{Here $N_j$ denotes the multiplicity of the spherical harmonics of order $j$.}\\
 	The computation builds on a characterization of the Morse index by means of a one dimensional singular eigenvalue problem, and is carried out by a detailed picture of the asymptotic behavior of both the solution and the singular eigenvalues and eigenfunctions.
 	In particular it is shown that nodal radial solutions have multiple blow-up at the origin, where each node converges (up to a suitable rescaling) to the bubble shaped solution of a limit problem.  
 	As side outcome we see that solutions are nondegenerate for $p$ near at $p_{\a}$, and we give an existence result in perturbed balls.
 	\end{abstract}

\tableofcontents

\section{Introduction}\label{sec:intro}
In this paper we continue the project started with \cite{AG-sez2} and use a singular eigenvalue problem to 
compute the Morse index of nodal radial solutions to semilinear equations. In particular we focus here on the problem 
\begin{equation}\label{H}
\left\{\begin{array}{ll}
-\Delta u = |x|^{\alpha}|u|^{p-1} u \qquad & \text{ in } B, \\
u= 0 & \text{ on } \partial B,
\end{array} \right.
\end{equation}
where $\a\geq 0$, $B$ stands for the unit ball in $\R^N$ in dimension $N\ge 3$, and $p>1$. 
When $\a>0$ problem \eqref{H} is known as the H\'enon problem, since it has been introduced by H\'enon in \cite{H} in the study of stellar clusters in radially symmetric settings, in 1973. Later on Ni, in the celebrated paper \cite{Ni}, proved the existence of a critical exponent related with the parameter $\a$, that we denote hereafter by 
\begin{equation}\label{p-alpha} 
p_{\a}= \frac{N+2+2\a}{N-2}
\end{equation}
which gives the threshold between existence and nonexistence of solutions. Using the fact that $H^1_{0,\rad}(B):=\{u\in H^1_0(B) : \text{ $u$ is radial}\}$ is compactly embedded in $L^{p+1}(B, |x|^\a dx)$ for every $1<p<p_\a$, Ni proved that \eqref{H} admits a positive radial solution, which is classical. The existence of radial solutions can be then extended to the case of nodal solutions with an arbitrary number of zeros (nodes) by means of a procedure introduced in \cite{BW93} and using again the compactness of the immersion of $H^1_{0,\rad}$ into $L^{p+1}$ as for the case of positive solutions. It is also possible to apply a uniqueness result of \cite{NN} to have that for any integer $m\ge 1$ there exists only a couple of radial solutions to \eqref{H} which have exactly $m$ nodal zones,  meaning that the set $\{x\in B \ : u(x)\neq 0\}$ has exactly $m$ connected components; they are one the opposite of the other and classical solutions (see, for instance, \cite[Proposition 5.14]{AG-sez2}). 

Moreover, a classical Pohozaev argument shows that the H\'enon problem \eqref{H} does not admit solutions when it is settled in a smooth bounded domain $\Omega$ which is starshaped with respect to the origin and $p\geq p_\a$. Then $p_\a$ exhibits the same role of the critical exponent $p^*=\frac{N+2}{N-2}$ for the Lane-Emden problem
\begin{equation}\label{LE}\left\{\begin{array}{ll}
-\Delta u = |u|^{p-1} u \qquad & \text{ in } B, \\
u= 0 & \text{ on } \partial B,
\end{array} \right.
\end{equation}
which corresponds to \eqref{H} in the case of $\a=0$. As we will see the relations between H\'enon and Lane-Emden problems are much  deeper. Indeed radial solutions to \eqref{H}  with $\a>0$ can be viewed as radially extended solutions to \eqref{LE} in a sense which will be clarified in Section \ref{sec:da-u-a-v}.\\
The H\'enon problem attracted the attention of many mathematicians since the paper \cite{SSW} in which the authors proved that the ground state solutions to \eqref{H}, namely solutions which minimizes 
the Energy functional
\[\mathcal E(u)=\frac 12 \int _B |\nabla v|^2-\frac 1{p+1}\int_B |x|^\a |u|^{p+1}\]
on the Nehari manifold
\[\mathcal N=\{v\in H^1_0(B)\ : \ \int_B|\nabla v|^2=\int_B |x|^\a |v|^{p+1}\}\]
for $1<p<p^*$ are nonradial provided $\a>0$ is sufficiently large. Nevertheless ground state solutions to \eqref{H} maintain a residual symmetry called foliated Schwartz symmetry, which appears in other similar contests in which the symmetry result of Gidas, Ni and Nirenberg in \cite{GNN} does not hold, namely in the case of annular domains or for nodal solutions. \\
Let us recall that the Morse index of a solution  $u$ to \eqref{H} is the maximal dimension of a subspace $X\subseteq H^1_0(B)$ where the quadratic form 
\[Q_u(\psi):=\int_B|\nabla \psi|^2-p|x|^\a|u|^{p-1}\psi^2\, dx\]
is negative defined.  The quadratic form $Q_u$ is associated with the linearized operator  in $B$ with Dirichlet boundary conditions 
\[L_u(\psi):=-\Delta \psi-p |x|^\a|u|^{p-1}\psi . \]
 Of course the Morse index can be computed counting (with multiplicity) the negative eigenvalues of $L_u$ in $H^1_0(B)$, but also some negative singular eigenvalues. This equivalence and the characterization of Morse index in terms of the singular eigenvalues of $L_u$ is given in details in \cite{AG-sez2} and will be essential for our aims. 
\\
It is well known that ground state solutions have Morse index one since they can be found as minima on the Nehari manifold, which has dimension one.   Then the result in \cite{SSW} tells that radial positive solutions to \eqref{H} can have Morse index greater than $1$, when $\a$ is large enough.
\\  
Starting from this consideration, in \cite{AG14} we computed the Morse index of radial positive solutions to \eqref{H} 
 showing that it converges to the value $1+N$ when $p\to p_\a$ and to the value $1$ as $p\to 1$, 
and we proved a first bifurcation result from the positive solution of the H\'enon problem which is, in our opinion, responsible of the symmetry breaking of \eqref{H}.
In this last paper a technical assumption, namely that $0<\a\leq 1$, is required  to deal with the linearized operator and compute the asymptotic Morse index of radial positive solutions.
	This assumption is removed here where, taking advantage from the analysis in \cite{AG-sez2} and
	using 
	a  singular eigenvalue problem associated to the linearized operator, the computation of the Morse index is performed for any value of $\a$. Nevertheless the result in \cite{AG14} put evidence on the fact that the symmetry breaking phenomenon pointed out in \cite{SSW}  is not related to large values of $\a$, but still holds when $0<\a\leq 1$.
\\
Later it has been proved in \cite{LWZ} that  the Morse index of any radial solution to \eqref{H} goes to $\infty$ as $\a\to \infty$, showing again the symmetry breaking of the  ground state solutions, for large values of $\a$. Their result has implications also concerning nodal ground state solutions, namely minima for $\mathcal E(u)$ on the nodal Nehari manifold
	\[ \begin{array}{rl}\mathcal N_{nod}=\Big\{v\in H^1_0(B) :&   v^+\neq 0, \ \int_B|\nabla v^+|^2=\int_B |x|^\a |v^+|^{p+1}, \\ &  v^-\neq 0 , \ \int_B|\nabla v^-|^2=\int_B |x|^\a |v^-|^{p+1}\Big\} .\end{array} \]
	Here $s^+$ ($s^-$) stands for  the positive (negative) part of $s$.
	As it is known by  \cite{BW} that they have Morse index $2$, the estimate in \cite{LWZ} implies that the symmetry breaking phenomenon concerns also  nodal ground state solutions. A similar consideration appears also in \cite{AG-sez2} as a consequence of some estimates on Morse index of radial nodal solutions, but only in the case of solutions which changes sign. 
	 
	The fact that the Morse index of any radial solutions to \eqref{H}  diverges as $\a\to \infty$ is a clue that the symmetry breaking phenomenon is not related with a nonradial solution whose energy is less than the radial one, but with infinitely many nonradial (nodal) solutions that should arise by bifurcation. 
	Indeed \cite{WY} found infinitely many positive multipeak  solutions, with arbitrarily large energy, when $p=p^*$ and in \cite{FN17} infinitely many positive solutions bifurcating from the radial positive solution when $p$ is near $p_\a$ are constructed.

In any case the exact Morse index of radial solutions to \eqref{H}, depending on the parameters $p$ and $\a$ and on the number of nodal zones $m$, is still unknown. To the authors' knowledge the unique results in this direction  are the computations in \cite{AG-sez2}, where a lower bound on the Morse index is presented and it is proved that the radial Morse index is equal to the number of nodal zones, namely the linearized operator $L_u$ has exactly $m$ negative eigenvalues whose related eigenfunction is radial.
\\ 
Beyond the symmetry breaking the interest of the mathematicians on the H\'enon problem \eqref{H} is due to the richness of its solutions set, which is completely different from the Lane Emden case.  For instance \cite{PS} produces multipeak solutions in the slightly subcritical range, by the Lyapunov-Schmidt reduction method. Moreover solutions appear also in a  critical and  supercritical range, namely when  $p=p^*$ or  $p>p^*$, and of course $p<p_\a$.  Concerning existence of nonradial solutions in the critical case we quote here \cite{S}  and the already mentioned \cite{WY}. Coming to the supercritical range, \cite{BS} produces nonradial positive solutions using minimization  in suitable symmetric spaces and \cite{C} produces positive solutions on perturbed balls for generic values of $p$, by a perturbation argument. Next for all values of the 
	exponent $p$ near at the threshold  $p_{\a}$, and any domain containing the origin, we mention \cite{GG} concerning existence of positive solutions and  also the papers \cite{CD17}  and \cite{CLP18}, where nodal bubble tower solutions are constructed  by a Lyapunov-Schmidt reduction method when $\a$ is not an even integer, respectively for $\a>0$ and $\a>-2$.
 
\

In this paper we want to fill the gap on the exact value of the Morse index of radial solutions to \eqref{H},  and, considering  $\a\geq 0$ as a fixed parameter we compute the Morse index of any radial solution to \eqref{H} in a left neighborhood of the critical exponent $p_\a$. To state our main result we denote by 
$\left[\frac{\alpha}{2}\right] =\max\left\{k\in {\mathbb N} \, : \, k\le\frac{\a}{2}\right\}$ the integer part of $\frac{\alpha}{2}$, and by 
$N_j =\frac{(N+2j-2)(N+j-3)!}{(N-2)!j!}$
the multiplicity of $\lambda_j=j(N+j-2)$ as an eigenvalue for the Laplace-Beltrami operator on the sphere ${\mathbb S}_{N-1}$. Moreover, understanding that for $\a=0$ a solution to \eqref{H} is exactly a solution to \eqref{LE} and $p_\a=p^*$, we can state:
\begin{theorem}\label{teo:morse-index}
	Let $u_p$ be any radial solution to \eqref{H} with $m$ nodal zones and let $\a\geq 0$. Then there exists $p^\star\in(1,p_\a)$ such that for any $p\in [p^\star,p_\a)$ we have 
	\begin{align} \label{morse-index-p-alpha}
	m(u_p) & = m \sum\limits_{j=0}^{1+\left[\frac{\alpha}{2}\right]} N_j  
	\intertext{as $\alpha>0$ is not an even integer, or}
	\label{morse-index-p-alpha-even}
	m(u_p)& =  m\sum\limits_{j=0}^{ \frac \a 2} N_j + (m-1) N_{1+\frac \a 2}. 
	\end{align}	
	if $\alpha=0$ or it is an even number.
\end{theorem}

\

This result is inspired to some previous papers on the Morse index of nodal radial solutions to the Lane Emden problem \eqref{LE} in dimension $N\geq 3$, see \cite{DIPN>3}  and in dimension $N=2$, see \cite{DIPN=2}, and to the possibility to obtain from its knowledge some existence results of nonradial nodal solutions  whose nodal set, namely $\{x\in B\ : u(x)=0\}$, does not touch the boundary of $B$, as in \cite{GI}. 
It is worth noticing that reading formula \eqref{morse-index-p-alpha-even} for $\alpha=0$ we get 
\[ m(u_p) =  m +(m-1)N\] for $p$ near to the critical exponent $p^*$, which is the exact formula obtained in \cite{DIPN>3} for solution to \eqref{LE}. As far as $\alpha\in(0,2)$, \eqref{morse-index-p-alpha} comes into play and the Morse index is larger, precisely  $m(u_p) = m (1+N)$, highlighting the fact that the Morse index increases with $\a$ and it changes corresponding exactly to the even values of $\a$.

\

{To have a precise idea of the Morse index of $u_p$ we observe that for small values of $\a$ these values are:
	\begin{align*}
	\begin{array}{ll}
	\a=0 & m(u_p)=m+(m-1)N\\
	0<\a<2 &  m(u_p)=m+mN\\
	\a=2 &m(u_p)=m+mN+(m-1)N_2\\
	2<\a<4& m(u_p)=m+mN+mN_2\\
	\a=4 & m(u_p)=m+mN+mN_2+(m-1)N_3
	\end{array}\end{align*}
	and so on, showing that the Morse index corresponding to the integer values of $\a$ is different from every other value for nodal solutions, i.e.~for $m\ge 2$. This seems to be a new phenomenon.

As mentioned before, Theorem \ref{teo:morse-index} brings new informations also in simplest case of positive solutions ($m=1$). In that case
formulas \eqref{morse-index-p-alpha} and \eqref{morse-index-p-alpha-even} can be written as
\begin{align} \label{morse-index-p-alpha-positive}
m(u_p) & =  \sum\limits_{j=0}^{k} N_j  \qquad \mbox{ as } 2(k-1)<\alpha \le 2k 
\end{align}	 
for $p$ near at $p_{\a}$.
\eqref{morse-index-p-alpha-positive} extends the computation made in \cite{AG14} for  $1<\a\le 2$ and describes the different values of the limit for larger values of $\a$.  As we have already remarked this last estimate was the crucial part for the bifurcation result in \cite{AG14}, since we have already noticed that the Morse index of  positive radial solutions converges to 1 as $p\to 1$. In a similar manner we expect that  formulas \eqref{morse-index-p-alpha} and \eqref{morse-index-p-alpha-even} are responsible of a nonradial bifurcation from nodal radial solutions to \eqref{H}, since the Morse index for $p$ near at $1$ has been computed in \cite{A18} obtaining
	\[  m(u_p)= 1+\sum_{i=1}^{m-1}\sum _{0\le j<J_i} N_j , \qquad \text{ as } J_i=\dfrac{(2+\alpha)\beta_i-(\n-2)}{2} .
	\]
	The parameters $\beta_i$ appearing here are linked to the zeros of the Bessel functions of first kind
	\[\jcal_{\beta}(r) = \sum\limits_{k=0}^{+\infty} \dfrac{(-1)^k} {k!\Gamma(k+1+\beta)} \left(\frac{r}{2}\right)^{2k+\beta}, \quad r\ge 0.\]
More precisely $\beta_i$ is characterized as the unique positive parameter for which the $i^{th}$ zero of ${\mathcal J}_{\beta_i}$ coincides with the $m^{th}$ zero of ${\mathcal J}_{\frac{\n-2}{2+\alpha}}$.
 Even though the values of the zeros of the Bessel functions (and therefore the parameters $\beta_i$) can be computed only by numerical approximations, in the same paper it is proved that $J_i> (2 + \a) (m-i)$, so that 
\[ \begin{split}
m(u_p)& \ge    1 + \sum\limits_{i=1}^{m-1}\sum\limits_{j=0}^{(2 + [\a]) (m-i)} N_j  
\\
&  = m+ (m-1) \sum\limits_{j=1}^{2+[\a]} N_{j} + (m-2) \sum\limits_{j=3+[\a]}^{2(2+[\a])} N_{j} \dots + \sum\limits_{j=(m-2)(2+[\a])+1}^{(m-1)(2+[\a])} N_{j} .
\end{split}\]
Starting from this estimate it is showed that the Morse index near at $p=1$ is greater than the one near at $p=p_{\a}$ for every  $m\ge 2$ and $N\ge 3$, therefore  a change in the Morse index appears at some values of $p$, and then we are confident that a bifurcation can arise. The fact that the eventual bifurcation is nonradial is then a consequence of Theorem 1.7 in \cite{AG-sez2}  which assures that any radial solution to \eqref{H} is radially nondegenerate for every $p$ and every $\a$. This has been observed in \cite{Weth} where a nonradial bifurcation from nodal radial solutions to \eqref{H}, with respect to the parameter $\a$ has been obtained.
            \\ 

Lastly we compare formulas \eqref{morse-index-p-alpha} and \eqref{morse-index-p-alpha-even} with the estimate from below of the Morse index obtained in Theorem 1.1 in \cite{AG-sez2}, which holds for any $p\in(1,p_{\a})$ and $\a\geq 0$ and states 
\begin{equation}\label{morse-estimate}
m(u_p) \ge   1 +	(m-1) \sum\limits_{j=0}^{1+[\frac{\alpha}{2}]}N_{j}  .\end{equation}
For positive solutions ($m=1$) it is known that this bound is optimal because the Morse index is equal to 1 when the exponent $p$ approaches the value 1.
For nodal solutions, in the case of Lane-Emden problem ($\a=0$)  in dimension $N\ge 3$, the estimate from below  is attained for $p$ near the critical exponent $p^*=\frac{N+2}{N-2}$ (see \cite{DIPN>3}). 
This is not the case anymore for the H\'enon problem, because the exact value obtained in Theorem \ref{teo:morse-index} overpasses the estimate from below.

\

Let us spend some words on how we get Theorem \ref{teo:morse-index}. First we exploit the characterization of the Morse index and the decomposition of some singular eigenvalues established in \cite{AG-sez2} and we relate the computation of the Morse index of any radial nodal solution, with $m$ nodal zones, to the knowledge of $m$ negative singular radial eigenvalues, see Proposition \ref{general-morse-formula-H}. Next we study their asymptotic behavior as $p\to p_\a$ together with the asymptotic profile of the associated eigenfunctions,  which is needed to deal with the last negative singular eigenvalue. 
This study furnishes immediately Theorem \ref{teo:morse-index} as a consequence of Proposition 1.5 and Theorem 1.7 in  \cite{AG-sez2}.
It also shows that  the bound  \eqref{morse-estimate} is obtained by estimating in a sharp way the singular radial eigenvalues: actually the first $m-1$ eigenvalues reach their upper bound for $p$ near at $p_{\a}$, giving the minimal contribution to the Morse index. In the Lane-Emden problem the contribution coming from the last eigenvalue is constant and therefore it does not influence the asymptotic behavior of the Morse index. On the contrary in the H\'enon problem the contribution of the last eigenvalue varies, and it is maximal for $p$ near at $p_\a$, minimal when $p$ is near at 1. It is thus clear that in the case of $\a=0$ the behavior of the $m-1$ singular negative eigenvalues is sufficient to compute the Morse index, while when $\a>0$ also the last negative eigenvalue comes into play and its estimate is the most difficult one.

\

The description of the asymptotic behavior of the singular eigenvalues and eigenfunctions relies on  the asymptotic analysis of the nodal radial solutions to \eqref{H} with $m$ nodal zones, which is indeed the second main aim of this paper. Let us remark that for the H\'enon problem the asymptotic profile is known only in the case of positive solutions.  Precisely \cite{AG14} describes the limit of the radial solution when $p\to p_{\a}$ and $\a$ is a fixed parameter, while \cite{BWang} study the limit  of both the radial and the ground state solution as $\a\to \infty$ and $p$ is fixed.
	\\
Here we are interested in the limit of the nodal radial solution when the exponent $p$ approaches the threshold $p_{\a}$, and to proceed with the further study of the related eigenvalues we need  to know the limit problem to which the solution converges and the behavior of its critical points and values. Concerning the Lane-Emden problem \eqref{LE} these topics have  been the subject of some interesting papers,  \cite{DIPN>3} and \cite{DIPN=2} among others. 
Solutions to \eqref{LE} indeed tend to concentrate in the origin  as showed in \cite{PW}, and admit a limit problem which can be used, for instance, to construct concentrating solutions in more general domains and with more general nonlinearities.
	This aspect is different when the dimension is $2$ (and $p\to \infty$) or  higher, so the two cases have to be treated separately.
	The H\'enon problem \eqref{H} shares the same duality: indeed when $N=2$  radial solutions exhibit a different limit problem and a different way to concentrate. For this reason we focus here on the case of $N\geq 3$ while we skip to the paper \cite{AG-N=2}, which brings to different conclusions, the study of the asymptotic behavior of $u_p$ and of its Morse index in the case of $N=2$.
	\\
To state the related result we need to introduce some notations.
Let $u_p$ be a radial solution with $m$ nodal zones and  

\

\begin{tabular}{l} 
		$0<r_{1,p}<r_{2,p}\dots <r_{m,p}=1$  be the zeros of $u_p$,\\[.2cm]
		$A_{i,p}$ the nodal zones of $u_p$, precisely \\[.2cm]
		$A_{0,p} = \{ x \, : \, |x|< r_{1,p}\}$ , and 
		$A_{i,p} = \{ x \, : \, r_{i,p}<|x|< r_{i+1,p}\}$  as $i=1,\dots m-1$, \\[.2cm]
		$\mu_{i,p}=\max\limits_{A_{i,p}} |u_{p}|$ the extremal value of $|u_{p}|$ in the   $(i+1)^{th}$ nodal zone $A_{i,p}$\\[.2cm]
		$\sigma_{i,p}\in A_{i,p}$  the extremal point of $|u_{p}|$ in the  $(i+1)^{th}$ nodal zone, \\[.2cm]
		so that $\mu_{i,p} = |u_{i,p}(\sigma_{i,p})|$, \\[.2cm]
		$\widetilde\mu_{i,p}  =\left(\mu_{i,p}\right)^{\frac{p-1}{2+\a}}$,\\[.2cm]
		 $\widetilde A_{i,p} = \left\{x \, : x/\widetilde\mu_{i,p} \in A_{i,p}  \right\}$ .
\end{tabular}

\

For every  $i=0,1,\dots m-1$ we  introduce the rescaled function 
\begin{align}\label{u-tilde-i} 
\widetilde{u}_{i,p}(x)  & := \frac{1}{\mu_{i,p}} \big|u_p\left( \frac x{\widetilde\mu_{i,p}}\right)\big| \quad \text{ as } x\in  \widetilde A_{i,p},  &  . 
\end{align}
Next, let  
\begin{equation}\label{U}
U_{\a}(x) :=  \left(1+\frac{|x|^{2+\a}}{(N+\alpha)(N-2)}\right)^{-\frac{N-2}{2+\a}} 
\end{equation}
be the unique radial  bounded solution of 
\begin{equation}\label{eq-U}
\left\{\begin{array}{ll}
-\Delta U_{\a}=|x|^{\a}U^{p_{\a}}_{\a}\qquad &\text{ in }\R^N, \\
U_{\a} > 0 &\text{ in }\R^N,\\
U_{\a}(0)=1,   &  \\
\end{array}\right.
\end{equation}
see the Appendix.
Of course when $\a=0$ \eqref{U} and \eqref{eq-U} give back the well known Talenti bubbles, which are related with problem \eqref{LE}.

Our main result on the asymptotic profile of radial solutions is the following: 
\begin{theorem}	\label{teo:conv-altre-zone-nodali}
   Let $u_p$ be any radial solution to \eqref{H} with $m$ nodal zones and $\a\geq 0$.
When $p\to p_{\alpha}$ we have
	\begin{align} \label{eq:massimi-pcritico}
	\mu_{i,p} & \to +\infty, & \text{ as } i=0,\dots m-1,\\
	\label{eq:convergenza-riscalata}\widetilde u_{0,p} & \to U_{\a}  \; \text{ in } C^1_{\mathrm{loc}}(\R^N) 
	\end{align}
	and whenever $m\geq 2$
	\begin{align}
	r_{i,p} & \to 0 , \qquad \sigma_{i,p}\to 0  , &  \text{ as } i=1, \dots m-1,\label{zeri-pcritico} \\
	\label{soluzione-pcritico}
	\widetilde u_{i,p} & \to U_{\a}  \quad \text{ in } C^1_{\mathrm{loc}}(\R^N\setminus\{0\})  & \text{ as } i=1,\dots m-1.  
	\end{align}
\end{theorem}

The statements concerning $i=0$ (i.e. the first nodal zone) follows easily by the already known results about the positive solution (see \cite{AG14}), while the ones concerning the following nodal zones are far more delicate. 
The main source of difficulty is the supercritical setting, which can be overtaken by performing a change of variable, introduced in \cite{GGN}, that allows to pass to a one-dimensional problem in a subcritical range. In this way the statement of Theorem \ref{teo:conv-altre-zone-nodali} becomes an extended radial version,  in a noninteger dimension, of the analogous one established in \cite{DIPN>3} for the Lane-Emden problem (i.e. when $\alpha=0$). At that point the most delicate part of the proof stands in establishing that the rescaled domains $\widetilde A_{i,p}$ invade  $\R^N$, and this step requests a very fine knowledge  of the speed of convergence (respectively, divergence) of the zeros (respectively, extremal values) of the solution. The proof presented here differs from the one in \cite{DIPN>3},  even in the case $\a=0$, because it does not rely on the a-priori knowledge of the bubble tower shape of the radial solution. Indeed from our approach it follows as a byproduct that  for any $\a> 0$ radial nodal solutions of the H\'enon problem have a bubble tower shape with multiple blow up at the origin.

\

Another interesting  consequence of the asymptotic analysis of the negative singular eigenvalues for the linearized operator $L_u$ and of the characterization of the degeneracy of radial solutions given in \cite{AG-sez2} is the following result:

\begin{theorem}\label{teo:nondegeneracy}
	Let $u_p$ be any radial solution to \eqref{H} with $m$ nodal zones and let $\a\geq 0$. Then there exists $\bar p \in(1,p_\a)$ such that $u$ is nondegenerate for any $p\in [\bar p ,p_\a)$.
\end{theorem} 
Let us recall that a solution $u$ is said nondegenerate whenever the linearized equation $L_u(v)=0$ does not admit nontrivial solutions in $H^1_0(B)$. This consideration is new, even in the  simpler case of the Lane-Emden problem, namely when $\a=0$, and extends a previous result in this direction in \cite{AG-sez2} where it was proved that $u$ is radially nondegenerate, namely that the linearized equation does not admit any radial solution.

To point out the usefulness of a nondegeneracy result as Theorem \ref{teo:nondegeneracy} we give here  an easy application in proving existence results.

\begin{theorem}\label{teo:esistenza}
Let $m\ge 1$ be any integer, $\a=0$ or $\a>1$, and  
\[\Omega_t:=\{x+t\sigma(x) : x\in B\},\] where $\sigma:\bar B\to \R^N$ is a smooth function, be a perturbation of the unit ball $B$. Then for every $p\in (\bar p,p_\a)$ problem \eqref{H} settled in $\Omega_t$ admits a classical solution with $m$ nodal zones, whose nodal set, when $m>1$, does not touch the boundary $\partial \Omega_t$ for $t$ small enough.
\end{theorem} 
In the authors' opinion the existence result in Theorem \ref{teo:esistenza} is interesting for two reasons. First because for $\a>1$ it inherits a supercritical range, where the lack of variational setting makes more difficult to obtain existence of solutions.
{ Secondly because it allows to construct solutions shaped as the radial solutions without requiring any symmetry on $\Omega$.} 

\

The paper is organized as follows. We start in Section \ref{sec:da-u-a-v} by proving the asymptotic profile of nodal radial solutions to \eqref{H} with $m$ nodal zones.  In Section \ref{se:3} we recall the characterization of the Morse index of nodal radial solutions and we relate its computation to the computation of the asymptotic limit of $m$ negative singular radial eigenvalues as $p\to p_\a$.  The analysis of the first $m-1$ ones, outlined in subsection \ref{3:m-1},  is based on the knowledge of the limit singular eigenvalue problem and to an estimate previously obtained in \cite{AG-sez2}. The major difficulty is the analysis of the last negative singular radial eigenvalue,  performed in subsection \ref{3:m}, which requires some fine estimates that extends the previous one in the case of $\a=0$. In Section \ref{se:4} we prove Theorem \ref{teo:nondegeneracy} and Theorem \ref{teo:esistenza}.  Lastly  we recall some well known fact about existence and uniqueness of solutions for the limit problems in the Appendix.


\

\section{The asymptotic profile of $u_p$ via a "radially extended" version of the Lane-Emden problem}\label{sec:da-u-a-v}
In this section we prove Theorem \ref{teo:conv-altre-zone-nodali} by relating radial nodal solutions to \eqref{H} with nodal solutions to a radially extended version of the Lane Emden problem and studying the asymptotic behavior of these radially extended solutions. In order to distinguish the two radial solutions to \eqref{H} we will denote hereafter by $u_p$ the nodal radial solution to \eqref{H} with $m$ nodal zones, that satisfies
\begin{equation}\label{segno-in-zero-u}
u_p(0)>0
\end{equation}
recalling that the other is given by the opposite of $u_p$.\\
The proof of Theorem \ref{teo:conv-altre-zone-nodali} will be given in a series of propositions in which we consider initially the first nodal zone, which is easier to handle, and then the case of the subsequent ones.

To begin with, we furnish the proof of Theorem \ref{teo:conv-altre-zone-nodali} for $i=0$, which is an immediate consequence of the asymptotic behavior of positive radial solutions in \cite{AG14} and does not rely on the radially extended Lane-Emden problem.

\begin{proof}[ Proof of Theorem \ref{teo:conv-altre-zone-nodali} for $i=0$]	
 Let us denote for a while by $u^m_p$ the nodal radial solution to \eqref{H} with $m$ nodal zones, that satisfies \eqref{segno-in-zero-u}.
	It suffices to notice that,  letting $r_{1,p}^m$ be the first zero of $u_p^m$, the scaled  function \[ \left(r^m_{1,p}\right)^{\frac{2+\a}{p-1}} u^m_p ( r^m_{1,p} x )\] 
	coincides with $u^1_p(x)$, the unique positive radial solution to the H\'enon problem in $B_1$. So applying \cite[Proposition 3.6]{AG14}  gives \eqref{eq:massimi-pcritico} and \eqref{eq:convergenza-riscalata} for $i=0$.	
\end{proof}

The investigation of subsequent nodal zones is more delicate. An useful tool is the change of variables
\begin{equation}\label{transformation-henon-no-c}
v(t)= \left(\frac{2}{2+\a}\right)^{\frac{2}{p-1}} u(r) , \qquad t= r^{\frac{2+\a}{2}},
\end{equation}
 which has been introduced in \cite{GGN} and transforms radial solutions to \eqref{H} into solutions of the radial extended Lane-Emden problem
\begin{equation}\label{LE-radial}
\begin{cases}
- \left(t^{M-1} v^{\prime}\right)^{\prime}= t^{M-1} |v|^{p-1} v  , \qquad  & 0< t< 1, \\
v'(0)=0 \ , \ v(1) =0  &
\end{cases}\end{equation}
where
\begin{equation}\label{M-a}
M=M(N,\alpha):  =\frac{2(N+\alpha)}{2+\alpha}
\end{equation}
plays the role of a noninteger dimension.
To deal with this problem we need to introduce the suitable functions spaces to which solutions to \eqref{LE-radial} belong. With this aim, for any $M,q\in\R$, $M\ge 2$ and $q\ge 1$,  we  let $L^q_M$ be the weighted Lebesgue space of measurable functions $v:(0,1)\to\R$ such that
\[ \int_0^1 r^{M-1} |v|^q dr < +\infty.\]
Next we denote by $H^1_{M}$ the subspace of $L^2_M$  made up by that functions $v$ which have weak first order derivative in $L^2_M$
 with \[ \int_0^1 r^{M-1} |v'|^2 dr < \infty, \]
 and
\begin{equation}\label{H-1-0-M}
H^1_{0,M} = \left\{ v\in H^1_M \, : \, v(1)=0\right\} \end{equation}
 which is  Hilbert space with the norm 
$$\|v\|^1_{M}:=\left(\int_0^1 r^{M-1}(v')^2\ dr\right)^\frac 12 $$
 due to a Poincar\'e inequality in the space $H^1_{0,M}$, see \cite[Lemma 6.1]{AG-sez2}.
The transformation \eqref{transformation-henon-no-c} maps  $H^1_{0,\rad}(B)$, the set of radial functions  in $H^1_0(B)$, into $H^1_{0,M}$ with  $M$ as in \eqref{M-a} and can be used in any dimension $N\geq 2$.
It allows to pass from $u_p$ (the radial solution to \eqref{H} with $m$ nodal zones satisfying \eqref{segno-in-zero-u}) to the solution to \eqref{LE-radial}
with $m$ nodal  zones satisfying 
	\begin{equation}\label{segno-in-zero-v}
v_{p}(0)>0 .
\end{equation}

The equivalence between radial solutions to \eqref{H} and solution to \eqref{LE-radial}, both in classical and weak sense and in any dimension $N\ge 2$, has been proved rigorously in \cite[Proposition 5.6 and Remark 5.12]{AG-sez2}. 
For the sake of completeness we recall that a weak radial solution to \eqref{H} can be seen as $u\in H^1_{0,N}$ such that 
\begin{equation}\label{H-radial-weaksol}
\int_0^1 r^{N-1} \left(u'\varphi'- r^{\a}|u|^{p-1}u \varphi\right) dr =0
\end{equation}
for any $\varphi\in H^1_{0,N}$, and similarly a weak solution to \eqref{LE-radial} is $v\in H^1_{0,M}$ such that
\begin{equation}\label{LE-radial-weaksol}
\int_0^1 t^{M-1} \left(v'\varphi'- |v|^{p-1}v \varphi\right) dt =0
\end{equation}
for any $\varphi\in H^1_{0,M}$.  In particular the same uniqueness result which holds for radial solutions to \eqref{LE} says that, for every integer $m\geq 1$, \eqref{LE-radial} admits a couple of solutions with $m$ nodal zones, which are one the opposite of the other and hence a unique solution $v_p$ which satisfies \eqref{segno-in-zero-v}.

Problem \eqref{LE-radial} can be seen as a "radially extended" version of the Lane-Emden problem since when $M$ is an integer $v_p$ actually is the radial nodal solution to the Lane-Emden problem 
\[ \left\{\begin{array}{ll}
-\Delta v = |v|^{p-1} v \qquad & \text{ in } B, \\
v= 0 & \text{ on } \partial B,
\end{array} \right.
\tag{\ref{LE}}\]
settled in the unitary ball of $\R^{M}$.
Also notice that when $N\geq 3$ then $M>2$ and the threshold exponent $p_{\a}$ of \eqref{H} can be expressed in term of the parameter $M=M(N,\alpha)$ as
\begin{align}
\label{p-critico-M}
p_\a=p_{M}=\frac{M+2}{M-2}. 
\end{align} 
For integer $M\ge 3$, $p_M$ is the critical value of the Sobolev immersion of $H^1_{0}(B)$ into $L^q(B)$, and it constitutes the threshold for the existence of solutions for \eqref{LE}. 
 For non integer $M>2$ the value $p_M$ is still the critical exponent for the immersion of $H^1_{0,M}$ into $L^{q}_M$   (see \cite[Lemma 6.4]{AG-sez2}) and constitutes again the threshold for the existence of solutions to \eqref{LE-radial}.
 
We will give the proof of Theorem \ref{teo:conv-altre-zone-nodali} in terms of the asymptotic behavior of the function $v_p$ as $p\to p_M$. For integer values of $M$ this has been proved in \cite[Propositions 3.3, 3.4 and Theorem 3.7]{DIPN>3}. Here we extend their result to any value of $M>2$.   

\

Let us first point out some qualitative property of the solutions $v_p$ that shall be useful in the sequel, namely
\begin{lemma}[Lemma 5.13 in \cite{AG-sez2}]\label{prop-4.2-H}
	Let $v _p\in H^1_{0,M}$ be the unique weak solution to \eqref{LE-radial} with $m$ nodal zones that satisfies \eqref{segno-in-zero-v}.  Then $v_p\in C^2[0,1]$ with 
	\[ v_p(0)= {\mathcal M}_{0,p} , \qquad v_p'(0) = 0.\] 
	Besides $v_p$ is strictly decreasing in its first nodal zone and it has a unique critical point, $s_{i,p}$ in any nodal domain. 
	In particular $s_{0,p}=0$ is the global maximum point  for $v_p$ and for $i=1,\dots m-1$ it holds
	\[{\mathcal M}_{0,p}=v_p(0) > {\mathcal M}_{1,p} =|v_p(s_{1,p})|> \dots {\mathcal M}_{m-1,p}=|v_p(s_{m-1,p})|.\]
	\end{lemma}

In order to study its asymptotic profile as  $p\to p_M$, we denote  hereafter  by
\begin{enumerate}[-]
\item $0<t_{1,p}<t_{2,p}\dots <t_{m,p}=1$   the zeros of $v_p$, 
\item $s_{0,p} =0$ the extremal point of $v_{p}$ in its first nodal  zone  $[0, t_{1,p})$,
\item $s_{i,p}$  the extremal point of $v_{p}$ in its  $(i+1)^{th}$  nodal  zone  $(t_{i,p}, t_{i+1,p})$  as $i=1,\dots m-1$, 
\item 
$\mathcal{M}_{i,p}= (-1)^i v_p(s_{i,p})$ 
 the extremal value of $|v_{p}|$ in the  $(i+1)^{th}$  nodal  zone, as $i=0,1,\dots m-1$,
\end{enumerate}
and,  letting $t_{0,p}=0$, we define the scaling
\begin{align}\label{v-tilde}
\widetilde{v}_{i,p}(t)  & = \frac{(-1)^i}{{\mathcal M}_{i,p}} v_p\left(\frac{t}{\widetilde {\mathcal M}_{i,p}}\right)   \quad \text{as }  t_{i,p}< \frac{t}{\widetilde {\mathcal M}_{i,p}} < t_{i+1,p}, 
\end{align} 
as $i=0,\dots, m-1$. Here	
\begin{equation}\label{m-tilde}
\widetilde {\mathcal M}_{i,p}  = \left({\mathcal M}_{i,p}\right)^{\frac{p-1}{2}}.
\end{equation}
These newly introduced items are related to the respective ones for the H\'enon problem  by the following relations 
\begin{enumerate}[-]
	\item $r_{i,p} = \left(t_{i,p}\right)^{\frac2{2+\a}}$ are the zeros of $u_p$,
	\item ${\mu}_{i,p} = \left(\frac{2+\a}2\right)^{\frac{2}{p-1}} \mathcal M_{i,p}$ are the local extremal values of $u_{p}$,
	\item $\sigma_{i,p} = \left(s_{i,p}\right)^{\frac2{2+\a}}$ are extremal values of $u_p$,
\end{enumerate}
\begin{equation}\label{u-v-tilde}
	\widetilde u_{i,p}(r)=\widetilde v_{i,p}\left( \frac{2}{2+\a} r^{\frac {2+\a}2}\right).
\end{equation}
 It is easy to check that, for $i=0,\dots, m-1$ the functions $\widetilde {v}_{i,p}$ solves
	\begin{equation}\label{eq:riscalate}
	\begin{cases}
	-(t^{M-1}\widetilde {v}_{i,p}')'=t^{M-1}\widetilde {v}_{i,p}^p, &\text{ for } t_{i,p}\widetilde {\mathcal M}_{i,p}<t<t_{i+1,p}\widetilde {\mathcal M}_{i,p}\\
	\widetilde {v}_{i,p}\left(t_{i+1,p}\widetilde {\mathcal M}_{i,p}\right)=0\\
	{\widetilde {v}_{i,p}\left(t_{i,p}\widetilde {\mathcal M}_{i,p}\right)=  0} &\text{ when } i\geq 1
	\end{cases}
	\end{equation}
	For simplicity we will assume that $\widetilde {v}_{i,p}$ is defined on $(0,\infty)$ extending at zero outside the interval $ (t_{i,p}\widetilde {\mathcal M}_{i,p},t_{i+1,p}\widetilde {\mathcal M}_{i,p})$.
	\\
	The main point of this section will be to show that the functions $\widetilde {v}_{i,p}$ admit for $i=0,\dots,m-1$ the following limit problem
\begin{equation}\label{equazione-limite-v}\begin{cases}
   -\left(t^{M-1} V'\right)' = t^{M-1} V^{p_M}, &  t>0 ,\\
V(t)>0 &  t>0 ,
\end{cases}\end{equation}
with the condition
\begin{equation}\label{eq:cond-V-in-zero}
V(0)=1
\end{equation}
whose unique weak solution in the space $\mathcal D_M(0,\infty)$ is given by 
\begin{equation}\label{soluzione-limite-v} 
 V_M(t)=\left(1+\frac{t^2}{M(M-2)} \right)^{-\frac{M-2}{2}}
\end{equation}
see the Appendix. 
Here $\mathcal D_M(0,\infty)$ stands for the closure of $C^{\infty}_0[0,\infty)$ under the norm
\[\int_0^\infty r^{M-1}|v'|^2\ dr, \]
which is a natural generalization of the space $D^{1,2}_{\rad}(\R^N)$ to the case of the non-integer dimension $M$, and by weak solution to 
\eqref{equazione-limite-v}  we mean a function $V\in \mathcal D_M(0,\infty)$ such that
\[\int_0^\infty r^{M-1}V'\varphi'\ dr=\int_0^\infty r^{M-1}V^{p_M}\varphi\ dr\]
for every $\varphi \in \mathcal D_M(0,\infty)$.

Since we have already proved that for $i=0$ the statements of Theorem \ref{teo:conv-altre-zone-nodali} hold true, it lasts to consider the case of $i\geq 1$. 
Concerning the subsequent nodal zones Theorem \ref{teo:conv-altre-zone-nodali} is equivalent to prove the two following propositions
\begin{proposition}	\label{conv-altre-zone-nodali-v}
	For any $M>2$, for any integer $m\ge 2$ and $i=1,\dots m-1$ we have	
	\begin{align} \label{massimi-pcritico-v}
	{\mathcal M}_{i,p}  \to +\infty,  
	\\ \label{zeri-pcritico-v}
	s_{i,p} \to 0 , \quad \quad t_{i,p} \to 0 ,
	\end{align}	
	as $p\to p_M$ given by \eqref{p-critico-M}.	
\end{proposition}

\begin{proposition}	\label{conv-sol-v}
	For any $M>2$, for any integer $m\ge 2$ and $i=1,\dots m-1$ we have
	\begin{align} \label{soluzione-pcritico-v-ogniM}
	\widetilde v_{i,p} \to V_M & \quad \text{ in } C^1_{\mathrm{loc}}(0,+\infty)
	\end{align} 
	as $p\to p_M$ given by \eqref{p-critico-M}. \end{proposition}

Indeed assuming Propositions \ref{conv-altre-zone-nodali-v} and \ref{conv-sol-v} one can easily deduce that the statement of Theorem \ref{teo:conv-altre-zone-nodali} holds true for $i=1,\dots m-1$.
\begin{proof}[Proof of Theorem \ref{teo:conv-altre-zone-nodali}  for $i=1,\dots m-1$]
	\eqref{massimi-pcritico-v} and \eqref{zeri-pcritico-v} immediately give \eqref{eq:massimi-pcritico} and \eqref{zeri-pcritico}, recalling the relations between $t_{i,p}$, $s_{i,p}$ and $\mathcal M_{i,p}$ and $r_{i,p}$, $\sigma_{i,p}$ and $\mu_{i,p}$. Similarly \eqref{soluzione-pcritico} follows by \eqref{soluzione-pcritico-v-ogniM} thanks to \eqref{u-v-tilde}.

\end{proof}

\subsection{The proof of Propositions \ref{conv-altre-zone-nodali-v} and \ref{conv-sol-v}}
In this subsection we will prove the two propositions which give the asymptotic behavior of the function $v_p$ as $p\to p_M$. 
Proposition \ref{conv-sol-v} will be proved passing to the limit into \eqref{eq:riscalate}, which is possible because 
	(i) the functions $\widetilde v_{i,p}$, extended to zero outside $(t_{i,p}\widetilde {\mathcal M}_{i,p},t_{i+1,p}\widetilde {\mathcal M}_{i,p})$, are uniformly bounded in $\mathcal D_M(0,\infty)$, and (ii) $t_{i,p}\widetilde {\mathcal M}_{i,p}\to 0$ while $t_{i+1,p}\widetilde {\mathcal M}_{i,p}\to \infty$.
	Item (ii) is the most delicate part of the proof and requests a deep knowledge of the behavior of the zeros and of the extremal values of the function $v_p$.
	Proposition \ref{conv-altre-zone-nodali-v} is a first step in this direction and it has been put in evidence because it has interest for itself.
	In any case the proof of these facts is quite involved and requires some preliminary estimates.

This first lemma provides a bound on the energy of the solution $v_p$ in each nodal zone and a bound on the first derivate of $v_p$ which will be useful in the sequel in order to pass to the limit into \eqref{eq:riscalate}. 

\begin{lemma}\label{alternativa}
	There exist  $\delta>0$ and  constant $C_1, C_2$ such that for every  $p\in (1+\delta , p_M)$
	\begin{align}\label{alternativa-3.16,17}
	\int_{t_{i-1,p}}^{t_{i,p}} t^{M-1} |{v_p}'|^2 dt &= \int_{t_{i-1,p}}^{t_{i,p}} t^{M-1} |v_p|^{p+1} dt \le C_1 , \\
	\intertext{for any $i=1,\dots m$ and  }
	\label{alternativa-lemma-3.10}
	|{v_p}'(t)|&\le {C_2}\, {t^{\frac {2-p(M-2)}{2} }} 
	\end{align}
	as $t\in (0,1)$.
\end{lemma}
\begin{proof}
	 Using as a test function in \eqref{LE-radial-weaksol} the function which coincides with $v_p$ on $(t_{i-1,p}, t_{i,p})$ and is zero elsewhere immediately gives the first  equality in \eqref{alternativa-3.16,17}. The subsequent estimate follows by the Nehari construction. 
	Indeed the solution $v_p$ can be produced by solving the minimization problem
	\begin{align*}
	\Lambda(t_1,\cdots t_{m-1}) & = \min\left\{  \sum\limits_{i=0}^{m-1} \inf\limits_{\phi_i\in{\mathcal N}(t_{i},t_{i+1})} {\mathcal E}(\phi_i) \, : \, 0=t_0<t_1<\cdots<t_m=1\right\},
	\intertext{where ${\mathcal N}(t_{i},t_{i+1})$ are the Nehari manifolds}
	{\mathcal N}(t_{i},t_{i+1})& =\left\{ \phi\in H^1_{0,M} \, : \, \phi(r)=0 \text { for $r$ outside $(t_{i},t_{i+1})$}, \right. \\
	& \qquad \qquad \left. \int_{t_{i}}^{t_{i+1}} r^{M-1} |\phi'|^2 dr=\int_{t_{i}}^{t_{i+1}} r^{M-1} |\phi|^{p+1} dr \right\} ,
	\intertext{and 	${\mathcal E}$ stands for the energy functional}
	{\mathcal E}(\phi)& = \frac{1}{2}\int_0^1 r^{M-1} |v'|^2 dr-\frac{1}{p+1}\int_0^1 r^{M-1} |v|^{p+1} dr.
	\end{align*}
	Afterwards it can be checked that choosing $t_1,\dots t_{m-1}$ which realize the minimum $\Lambda$ and gluing together, alternatively, the positive and negative solution in the sub-interval $(t_{i-1}, t_{i})$, gives a nodal solution to \eqref{LE-radial}, which by uniqueness, see \cite{NN}, coincide with $v_p$ up to the sign.
	We refer to \cite{BW93} or \cite[Sec. 5.3]{AG-sez2} for more details. 
	To the current purpose  it suffices to notice that for all $i=0,\dots m-1$ the restrictions $v_{i,p}$ of the solution $v_p$ to its nodal zones $(t_{i},t_{i+1})$ belong to the Nehari sets ${\mathcal N}(t_{i},t_{i+1})$ and therefore 
	\begin{align*} 
	\int_0^1 r^{M-1} |v'_{p}|^2 dr	& = \sum\limits_{i=0}^{m-1} \int_{t_{i}}^{t_{i+1}} r^{M-1} |v'_{i,p}|^2 dr 
	=\dfrac{2(p+1)}{p-1}\Lambda(t_1,\cdots t_{m-1}) \\ &\le  \dfrac{2(p+1)}{p-1} \sum\limits_{i=0}^{m-1} {\mathcal E}(\phi_{i,p}) 
	= \sum\limits_{i=0}^{m-1} \int_{\frac{i}{m}}^{\frac{i+1}{m}} r^{M-1} |\phi'_{i,p}|^2 dr 
	\end{align*}
	for every $m$-ple of functions $\phi_{i,p} \in {\mathcal N}(\frac{i}{m},\frac{i+1}{m})$ and for every $p\in(1,p_{M})$. So \eqref{alternativa-3.16,17} can be proved by producing 
	a sequence $\phi_{i,p} $ with 
	\begin{equation}\label{phi-i-p}
	\lim\limits_{p\to p_{M}}  \int_{\frac{i}{m}}^{\frac{i+1}{m}} r^{M-1} |\phi'_{i,p}|^2 dr <+\infty \qquad \text{ as } i=0,\dots m-1 .
	\end{equation}
	To this aim we take continuous piecewise linear functions defined as
	\[\phi_{i,p}(r)=\begin{cases} a_{i,p}(r-\frac{i}{m}) \quad & \text{ as }  \frac{i}{m} <r\le \frac{2i+1}{2m} , \\
	a_{i,p}(\frac{i+1}{m}-r) \quad & \text{ as }   \frac{2i+1}{2m} <r<\frac{i+1}{m} , \\
	0 & \text{ elsewhere}
	\end{cases}\]
	and pick $a_{i,p}>0$ in such a way that $\phi_{i,p} \in {\mathcal N}(\frac{i}{m},\frac{i+1}{m})$.
	Since
	\begin{align*}
	\int_{\frac{i}{m}}^{\frac{i+1}{m}} r^{M-1} |\phi'_{i,p}|^2 dr & =a_{i,p}^2 \int_{\frac{i}{m}}^{\frac{i+1}{m}} r^{M-1} dr = \dfrac{a_{i,p}^2} {m^M}  \int_0^1(i+r)^{M-1} dr,
	\\ 
	\int_{\frac{i}{m}}^{\frac{i+1}{m}} r^{M-1} |\phi_{i,p}|^{p+1} dr & =a_{i,p}^{p+1} \left(\int_{\frac{i}{m}}^{\frac{2i+1}{2m}} r^{M-1}(r-\frac{i}{m})^{p+1} dr +  \int_{\frac{2i+1}{2m}}^{\frac{i+1}{m}} r^{M-1}(\frac{i+1}{m}-r)^{p+1} dr \right) \\
	& = \frac{a_{i,p}^{p+1}}{m^{M+p+1}} \int_{0}^{\frac{1}{2}}\left( (i+r)^{M-1}+(i+1-r)^{M-1} \right) r^{p+1} dr ,
	\end{align*}
	$\phi_{i,p} \in {\mathcal N}(\frac{i}{m},\frac{i+1}{m})$ provided that
	\[ a_{i,p}^{p-1}= \dfrac{m^{p+1}\int_0^1(i+r)^{M-1}dr}{\int_{0}^{\frac{1}{2}}\left( (i+r)^{M-1}+(i+1-r)^{M-1} \right) r^{p+1} dr } ,
	\]
	and in that case
	\[\int_{\frac{i}{m}}^{\frac{i+1}{m}} r^{M-1} |\phi'_{i,p}|^2 dr = \dfrac{m^{2\frac{p+1}{p-1}-M} \left(\int_0^1(i+r)^{M-1}dr\right)^{\frac{p+1}{p-1}}}{\left(\int_{0}^{\frac{1}{2}}\left( (i+r)^{M-1}+(i+1-r)^{M-1} \right) r^{p+1} dr \right)^{\frac{2}{p-1}}}
	,\]
	 which clearly yields \eqref{phi-i-p}.

	Besides from \eqref{alternativa-3.16,17} and the Talenti's Sobolev embedding for the spaces $H^1_{0,M}$ stated by \cite[Lemma 6.3]{AG-sez2} we also get 
	\begin{align*}
	&  \big(\int_0^1 t^{M-1} |v_p|^{\frac{2M}{M-2}} dt \big)^{\frac 2{2^*_M}}\le S_M \int_{0}^{1} t^{M-1} |{v_p}'|^2 dt\leq C.
	\end{align*} 
	Next integrating equation \eqref{LE-radial} on $(0,t)$ and using that $v_p'(0)=0$  give 
	\begin{align*}
	|{v_p}'(t)| &  \le \frac{1}{t^{M-1}} \int_0^t \tau^{M-1} |v_p|^p d\tau .
	\end{align*}
	Eventually Holder inequality yields
	\begin{align*}
	|{v_p}'(t)| &  \le \frac{1}{t^{M-1}} \left(\int_0^t \tau^{M-1} |v_p|^{\frac{2M}{M-2} }  d\tau \right)^{\frac{p(M-2)}{2M}}
	\left(\int_0^t \tau^{M-1} d\tau \right)^{1-\frac{p(M-2)}{2M}} \\
	\le &
	\frac{1}{t^{M-1}} C  t^{M-\frac{p(M-2)}{2}} =  {C}{t^{1-\frac{p(M-2)}{2}}} .
	\end{align*}
\end{proof}

Next lemma shows that the energy of $v_p$ is bounded also from below in each nodal zone, and so ensures that the local extremal values do not vanish.   

\begin{lemma}\label{claim-1-facile} For all $i= 0 ,\dots m-1$ we have 
	\[
	\liminf\limits_{p\to p_M} \int_{t_{i,p}}^{t_{i+1,p}}t^{M-1}  |v_{p}|^{p+1} dt = \liminf\limits_{p\to p_M} \int_{t_{i,p}}^{t_{i+1,p}}t^{M-1} |v_{p}'|^2   dt\ge S_M^{\frac{M}{2}}.
	\]
	In particular $\liminf\limits_{p\to p_M} {\mathcal M}_{i,p} >0$. 
\end{lemma}
	Here $S_M$ is the best constant for the Sobolev embedding of $H^1_{0,M}$ into $L^{2^\star_M}_M$, with $2^\star_M= \frac{2M}{M-2}$ (see \cite[Lemma 6.3]{AG-sez2}). 
\begin{proof}
	Since $\lim\limits_{p\to p_M} \frac{p+1}{p-1} = \frac{M}{2}$, it suffices to show that
	\[\liminf\limits_{p\to p_M} \left(\int_{t_{ i,p}}^{t_{ i+1,p}}t^{M-1} (v'_{p})^2  dt\right)^{\frac{p-1}{p+1}}\ge S_M.
	\]
	We use as a test function in \eqref{LE-radial-weaksol} the function $v_{i,p}$ which coincides with $v_p$ in the set $(t_{ i,p}, t_{ i+1,p})$ and it is zero elsewhere, obtaining           that
	\[ \int_{t_ {i,p}}^{t_ {i+1,p}}t^{M-1}  (v'_{p})^2dt = \int_0^1t^{M-1} (v'_{i,p})^2 dt = \int_0^1t^{M-1} |v_{i,p}|^{p+1} dt = \int_{t_{ i,p}}^{t_{ i+1,p}}t^{M-1}  |v_{p}|^{p+1} dt.
	\]
	Hence
	\begin{align*}
	\left(\int_{t_{ i,p}}^{t_{ i+1,p}}t^{M-1} (v'_{i,p})^2 dt\right)^{\frac{p-1}{p+1}} = \frac{\int_0^1t^{M-1} (v'_{i,p})^2 dt }{\left(\int_0^1t^{M-1} |v_{i,p}|^{p+1} dt \right)^{\frac{2}{p+1}}} \underset{\text{Holder}}{\ge} \\
	\frac{\int_0^1t^{M-1} (v'_{i,p})^2 dt }{M^{\frac{2}{2^{\ast}_M} - \frac{2}{p+1}}\left(\int_0^1t^{M-1} |v_{i,p}|^{2^{\ast}_M} dt \right)^{\frac{2}{2^{\ast}_M}}} 
	\ge  M^{\frac{2}{p+1}-\frac{2}{2^{\ast}_M}} S_M,
	\end{align*}
	where the last inequality holds thanks to the  Talenti's Sobolev embedding, \cite{talenti}, see also \cite[Lemma 6.3]{AG-sez2}. The first part of the claim  follows because $\frac{2}{p+1}-\frac{2}{2^{\ast}_M}\to 0$.
	\\
 To conclude the proof it suffices to notice that, due to Lemma \ref{prop-4.2-H},
	\begin{equation}\label{T-no-1} (m-i) S_M^{\frac{M}{2}} \le \liminf\limits_{p\to p_M} \int_{t_{i,p}}^{1}t^{M-1} |v_{p}(t)|^{p+1} dt 
	\le 
	\liminf\limits_{p\to p_M} ({\mathcal M}_{i,p})^{p+1} (1-t_{i,p}).
	\end{equation}
\end{proof}

As a corollary of the previous lemmas we obtain the boundedness of $\widetilde v_{i,p}$ in $\mathcal{D}_M(0,+\infty)$.

\begin{corollary}\label{corollario}
	For $i=0,\dots,m-1$ let $\widetilde v_{i,p} $ be the rescaled function defined in \eqref{v-tilde} and extended to zero outside $(t_{i,p}, t_{i+1,p})$. Then there exists $\delta>0$ and a constant $C_3$ such that 
	\begin{equation}
	 \int_0^{\infty}	t^{M-1}\left(\widetilde v_{i,p} '\right)^2 dt  \leq C_3
	\end{equation}
	for every $p\in (p_M-\delta,p_M)$.
\end{corollary}
\begin{proof}
	It is enough to observe that
	\begin{align*}
	\int_0^{\infty}	t^{M-1}\left(\widetilde v'_{i,p} \right)^2 dt  & 
=   \int_{t_{i-1,p}\widetilde {\mathcal M}_{i,p}}^{t_{i,p}\widetilde {\mathcal M}_{i,p}} t^{M-1}\left(\widetilde v_{i,p} '\right)^{ 2} dt\\
	& ={\mathcal M}_{i,p}^{\frac{p-1}{2}(M-2) - 2} \int_{t_{i-1,p}}^{t_{i,p}} t^{M-1} |v'_p|^2 dt \le C_3	
	\end{align*}
 by \eqref{alternativa-3.16,17},	since $\frac{p-1}{2}(M-2) < 2$ and $\mathcal{M}_{i,p}\ge \e>0$ by Lemma \ref{claim-1-facile}.
\end{proof}

We also recall a fine estimate of the behavior of the function $v_p$ in a left neighborhood of its zeros,  which is fundamental in the computations.
\begin{lemma}\label{prop-3.6-DIPN>3}
	\[|v_p(t)| \le \dfrac{\mathcal M_{0,p}}{\left(1 + \frac{(\widetilde{\mathcal{M}}_{0,p} t)^2}{M(M-2)}\right)^{\frac{M-2}{2}}}  \] 
	for every $0\leq t < t_{1,p}$.
	\\ Moreover if $s_{i,p}/t_{i+1,p}\to 0$ for some $i=1,\dots m-1$, 
	then for any $\varepsilon \in (0,1)$ there exist $\gamma=\gamma(\varepsilon)>1$ and $\bar p=\bar p(\varepsilon)<p_M$ such that
	\begin{equation}\label{eq:stima-v}
	|v_p(t)| \le \dfrac{\mathcal M_{i,p}}{\left(1 + \frac{\varepsilon (\widetilde{\mathcal{M}}_{i,p} t)^2}{M(M-2)}\right)^{\frac{M-2}{2}}} \end{equation}
	for every $\gamma s_{i,p} \le t \le t_{i+1,p}$ and $p\in (\bar p, p_M)$.
\end{lemma}
The first part of the statement, concerning the first nodal zone, can be proved by performing the Emden-Fowler transformation and following the line of \cite{AP86}, see also \cite[Lemma 2]{FN17}, where the same estimate is obtained for positive solutions. Next their arguments can be adapted to deal with the following nodal zones, as it has been done in \cite[Propositions 3.5 and 3.6]{DIPN>3}, where the same statement of Lemma \ref{prop-3.6-DIPN>3} was proved only for integer $M$. Their proof applies to any $M>2$ because it only makes use of ODE arguments. 

\

 Let us remark that the previous estimates can be read in terms of the scaled functions $\widetilde v_{i,p}$ as follows
\begin{corollary}\label{corollario2}	
	\[\widetilde v_{0,p}(t)\leq V_M(t) \quad \mbox{ for every } 0\leq t<t_{1,p}\widetilde {\mathcal M}_{0,p}. \]
 Moreover  if $s_{i,p}/t_{i+1,p}\to 0$ for some $i=1,\dots m-1$,  then for any $\varepsilon \in (0,1)$ there exist $\gamma=\gamma(\varepsilon)>1$ and $\bar p=\bar p(\varepsilon)<p_M$ such that
 \begin{equation}\label{eq:stima-v-tilde}
\widetilde v_{i,p}(t)\leq V_M(\sqrt \e t)  \quad \mbox{ for every } \gamma s_{i,p}\widetilde {M}_{i,p}< t<t_{i+1,p}\widetilde {M}_{i,p}
\end{equation}
as $p\in (\bar p, p_M)$.
\end{corollary}

Proposition \ref{conv-altre-zone-nodali-v} will be proved proceeding forward from the first nodal zone to the second one and so on. Hence the starting point stands in describing the asymptotic of $\widetilde v_{0,p}$ in the first nodal zone, which is a consequence of Theorem \ref{teo:conv-altre-zone-nodali} for $i=0$ and has been already proved. 
	Precisely the part of the statement concerning the first nodal zone is equivalent to 

\begin{proposition}\label{prop:conv-1-zona-v}
	For every $M>2$ and any integer $m\ge 1$, ${\mathcal M}_{0,p}  \to +\infty$ and
	$\widetilde v_{0,p}  \to V_M$ in $C^1_{\mathrm{loc}}[0,+\infty)$, as $p\to p_M$.
\end{proposition}
\begin{proof}
	It suffices to take $\alpha>0$ such that $N=M+\alpha(M/2-1)$ is an integer and then apply Theorem  \ref{teo:conv-altre-zone-nodali}, which has already been proved at the beginning of this section in the particular case $i=0$. 
\end{proof}

It will also be needed to establish relations between  the asymptotic of the extremal values in different nodal zones.
To this aim we introduce another scaling of the solution $v_p$ that we will use later on, precisely 
\begin{equation}\label{def-w} w_{i,p} (r) =
(t_{i,p})^{\frac{2}{p-1}} v_p(t_{i,p}\, r) , 
\end{equation}
which satisfies 
\begin{equation}\label{eq-w} \begin{cases}
-\left( r^{M-1} w_{i,p}' \right)' = r^{M-1}|w_{i,p}|^{p-1}w_{i,p} \quad & \text{ as } 0< r < 1/t_{i,p} , 
\\
w'_{i,p}(0)= w_{i,p}(1)= 0 = w_{i,p}(1/t_{i,p}). &
\end{cases}\end{equation}
We therefore see that $w_{i,p}$ on the interval $(0,1)$ coincides with the nodal solution to \eqref{LE-radial} which has exactly $i$ nodal zones, but is defined also in the larger interval $(0, 1/t_{i,p})$. This will be of help when deducing the asymptotic of the extremal value in the $i^{th}$ nodal zone from the one in the previous nodal zone.
We deal by now with the behavior of the function  $w_{i,p}$ to the left of $r=1$.

\begin{lemma}\label{lemma-w'(1)=0}
	Take $i=1,\dots m-1$ and assume that, for a sequence $p_n\to p_M$,
	\begin{align}
	\label{hp-int} \tau_{n}:= {s_{i-1,p_n}}/{t_{i,p_n}}\to 0 \\
	\label{hp-est}
	\rho_n:=t_{i,p_n}\widetilde{\mathcal M}_{i-1,p_n} \to +\infty.
	\end{align}
	Then $w_{i,p_n}\to 0$ uniformly in any set $[1-\delta,1]$ for $0<\d<1$. 
\end{lemma}
\begin{proof} 
	For simplicity of notation we shall write $w_{n}$ and $t_{n}$ instead of $w_{i,p_n}$ and $t_{i,p_n}$.
	By Lemma \ref{prop-3.6-DIPN>3}, for a fixed $\varepsilon>0$  there exists $\gamma$ such that 
	\begin{equation}\label{3.23-w}
	|w_n(r) |\le \dfrac{\rho_n^{\frac{2}{p_n-1}}} {\left(1 + \frac{\varepsilon(\rho_n r)^2}{M(M-2)}\right)^{\frac{M-2}{2}}} \quad \mbox{ for } \gamma \tau_n\le r \le 1 .
	\end{equation}
	If $\delta \in (0,1)$ is fixed, by hypothesis \eqref{hp-int} there exists $\bar n$ such that $\gamma\tau_n\le 1-\delta$ if $n\ge \bar n$ and \eqref{3.23-w} implies that for any $r\in [1-\delta, 1]$ we have 
	\[ |w_n(r)| \le C(\varepsilon,\delta) \rho_n^{\frac{2}{p_n-1}-M+2} =o(1)\]  as $n\to \infty$ by \eqref{hp-est}, because $\frac{2}{p_n-1}-M+2\to -\frac{M-2}{2}<0$.  
\end{proof}

We are now in the position to prove Proposition \ref{conv-altre-zone-nodali-v}.

\begin{proof}[Proof of Proposition \ref{conv-altre-zone-nodali-v}]
	It is worth noticing by now that the Radial Lemma in $H^1_{0,M}$  (see \cite[Lemma 6.2]{AG-sez2}) yields
	\[ t_{i,p}< s_{i,p} \le  \left(\frac{\int_0^1t^{M-1}\left|v_p'(t)\right|^2 dt}{(M-2)\left({\mathcal M}_{i,p}\right)^2}\right)^{\frac{1}{M-2}}
	\  \underset{\eqref{alternativa-3.16,17}}{\le} \frac{C}{\left({\mathcal M}_{i,p}\right)^{\frac{2}{M-2}}}.
	\]
	So, once \eqref{massimi-pcritico-v} has been established, then both $t_{i,p}$ and $s_{i,p}$ go to zero, which means that the proof is completed. 
	Besides it is already known  by Proposition \ref{prop:conv-1-zona-v} that $\mathcal{M}_{0,p}\to +\infty$, therefore \eqref{massimi-pcritico-v} can be proved by taking that $\mathcal{M}_{i-1,p}\to +\infty$ and deducing that also $\mathcal{M}_{i,p}\to +\infty$.
	To this aim we assume by contradiction that ${\mathcal M}_{i,p}$ is bounded, so that the functions $v_p$ are uniformly bounded in $[ t_{i,p},1]$ by  Lemma \ref{prop-4.2-H}.
        Up to an extracted sequence we may take that  ${\mathcal M}_{i,p}\to \bar{\mathcal M}\in (0,+\infty)$ and $t_{i,p}\to T\in[0,1)$. Indeed the occurrence  $\bar{\mathcal M}=0$  is ruled out by Lemma \ref{claim-1-facile}, and $T=1$ is not allowed by  \eqref{T-no-1}  since we are assuming $\mathcal {M}_{i,p}$ bounded.
	Next we argue separately according if
	\begin{enumerate}[a)]
		\item $T=0$,
		\item	$T\in (0,1)$.
	\end{enumerate}
        In case a) we observe that  the functions $v_p$ are bounded in $H^1_{0,M}$ by  \eqref{alternativa-3.16,17}.
      So, up to a subsequence, $v_p$ converges to a function $\bar v$ weakly in $H^1_{0,M}$, and also strongly in $L^q_M$ for every $1<q< \frac{2M}{M-2}$ by the compact Sobolev embedding  stated in \cite[Lemma 6.4]{AG-sez2} . It is thus easy to see that  we can pass to the limit in \eqref{LE-radial-weaksol} so that $\bar v\in H^1_{0,M}$ is a weak solution to 
\[\begin{cases}
-\left(t^{M-1} {\bar v}'\right)' =t^{M-1}  |\bar v|^{p_M-1} \bar v   &   t\in (0 , 1) , \\
\bar v(1) =0. & 
\end{cases}\]
Next we denote by $\hat v_{i,p}$ the function which coincides with $v_p$ on $(t_{i,p}, 1)$ and is null on $[0,t_{i,p}]$. Since we are assuming that ${\mathcal M}_{i,p}$ remains bounded, Lemma \ref{prop-4.2-H} assures that $\hat v_{i,p}$ is uniformly bounded on $[0,1]$ and clearly it converges pointwise a.e. to $\bar v$ because we are taking that $t_{i,p}\to 0$. So we can pass to the limit and compute
\begin{align*}
\int_0^1 t^{M-1}|\bar v|^{p_M+1} dt & =  \lim\limits_{p\to p_M}  \int_{0}^1 t^{M-1}|\hat v_{i,p}|^{p+1} dt  \\
&=\lim\limits_{p\to p_M}  \int_{t_{i,p}}^1 t^{M-1}|v_p|^{p+1} dt 
 \ge (m-i) S_M^{\frac{M}{2}}
\end{align*}
by Lemma \ref{claim-1-facile}. Hence $\bar v$ is not trivial.
Eventually performing the change of variables \eqref{transformation-henon-no-c} backwards (and invoking \cite[Proposition 5.6]{AG-sez2}) gives a nontrivial radial solution of the H\'enon problem in a ball with the exponent $p_{\alpha}$, which is not possible by Pohozaev identity.

	In case b), we look at the function $w_{i,p}$ introduced in \eqref{def-w}. 
 In the present setting $\tau_p= s_{i-1,p}/t_{i,p}\to 0$  and $\rho_p = t_{i,p} \widetilde{\mathcal M}_{i-1,p}\to \infty$ (since we are assuming $s_{i-1,p}\to 0$, $t_{i,p}\to T\neq 0$, $\widetilde{\mathcal M}_{i-1,p}\to \infty$),	 so  Lemma \ref{lemma-w'(1)=0} implies that $w_{i,p}\to 0$ uniformly on any set of type $[1-\d,1]$. In particular $w_{i,p}$ is uniformly bounded on $[1-\d,1]$. But the same holds also in the set $[1,1/t_{i,p}]$, because in that case 
			\[ |w_{i,p}(r)| = t_{i,p}^{\frac{2}{p-1}} |v_p(t_{i,p} r )| \le \mathcal M_{i,p} \]
				Moreover  
	\begin{equation}\label{eq:w_p}
	|w_{i,p}'(r)|\leq |v_p'(t_{i,p} r)|\leq C \ \ \ \text{ in }(1-\d,1/t_{i,p})\end{equation}
	thanks to \eqref{alternativa-lemma-3.10}, since we are assuming that $t_{i,p}$ does not vanish. 
	Next using the fact that $w_p$ is a classical solution to \eqref{eq-w} one sees that also $|w''_{i,p}|\leq C$  in $(1-\d,1/t_{i,p})$ so that, up to a subsequence, 
	 $w_{i,p}$ converges in 
	$C^1(1-\d,1/T)$ to  a function $ w$ which weakly solves 
	\[\begin{cases}
	-\left( t^{M-1} { w}' \right)' = t^{M-1}| w|^{p_M-1}  w \quad & \text{ as } 1-\e< t < 1/T , \\
	w(1)= 0 = w(1/T). &
      \end{cases}\]
     Next, reasoning as in \cite[Lemma 5.2]{AG-sez2} one can see that  a weak solution $w$ is also classical.
	As we already noticed that $w=0$ on the interval $(1-\d,1)$, the unique continuation principle gives that $w$ is identically zero. But this contradicts   Lemma \ref{claim-1-facile} since  by the boundedness of $w_{i,p}$
	\begin{align*}
	0 & = \int_1^{1/T} t^{M-1}| w|^{p_M+1} dt = \liminf\limits_{p\to p_M} \int_{1}^{1/t_{i,p}}
	t^{M-1}|w_{i,p}|^{p+1} dt \\
	& = \liminf\limits_{p\to p_M} \, (t_{i,p})^{\frac{2(p+1)}{p-1}-M} \int_{t_{i,p}}^1 t^{M-1}|v_p|^{p+1} dt 
	\ge (m-i) S_M^{\frac{M}{2}} 
	\end{align*}
	because $\frac{2(p+1)}{p-1}-M\to 0$ and $t_{i,p}\to T \in (0,1)$.
		So neither item b) can happen and the proof is completed.
\end{proof}

Eventually we prove  Proposition \ref{conv-sol-v}. 

\begin{proof}[Proof of Proposition \ref{conv-sol-v}]
Assume for a while to know that
\begin{align}\label{hp-a}
t_{i+1,p}\,\widetilde {\mathcal M}_{i,p}  \to +\infty ,
\\  \label{hp-b}
s_{i,p}\,\widetilde {\mathcal M}_{i,p} \to  0 ,
\\  \label{hp-c}
t_{i,p}\,\widetilde {\mathcal M}_{i,p} \to  0 ,
\end{align}
 as $p\to p_M$, for $i=1,\dots m-1$.
Then it is not hard conclude the proof.
As the nodal domain $(t_{i,p}\,\widetilde {\mathcal M}_{i,p} , t_{i+1,p}\,\widetilde {\mathcal M}_{i,p})$ invades $(0,+\infty)$, it is equivalent to prove the convergence of the sequence of functions  $\widetilde v_{i,p}$ extended to be zero outside $(t_{i,p}\,\widetilde {\mathcal M}_{i,p} , t_{i+1,p}\,\widetilde {\mathcal M}_{i,p})$
so that they belong to $\mathcal D_M(0,\infty)$. 
We recall that $\widetilde v_{i,p}$ is  nonnegative and solves the equation \eqref{eq:riscalate}
in classical sense. Moreover its norm in $\mathcal D_M(0,\infty)$ is bounded  (uniformly w.r.t.~$p$) by Corollary \ref{corollario}
therefore $\widetilde v_{i,p}$ converges to a function $\widetilde v$ weakly in ${\mathcal D_M(0,\infty)}$, strongly in $L^q(0,\infty)$ as $q= 2^{\star}_M$ and pointwise a.e., up to an extracted sequence.

We can then pass to the limit in the weak formulation of \eqref{eq:riscalate}, provided that  the functions $r^{M-1}{\widetilde v_{i,p}}^p$ are uniformly dominated by a function 
in $L^1(0,\infty)$ (for $p$ near to $p_M$).
{ First observe that we can apply Corollary \ref{corollario2} thanks to assumptions \eqref{hp-a} and \eqref{hp-b}. More precisely  we know that for a fixed $\e>0$ there exist $\gamma>0$ and $\bar p \in (1, p_M)$ such that for every  $p\in (\bar p, p_M)$ and $r\in(\gamma s_{i,p}\widetilde{\mathcal M}_{i,p}, t_{i+1,p}\widetilde{\mathcal M}_{i,p})$
  \[\widetilde v_{i,p}(r) \leq V_M(\sqrt \e r)\]
  and, recalling that $\widetilde v_{i,p}=0$ when $r>t_{i+1,p}\widetilde{\mathcal M}_{i,p}$ and $\gamma s_{i,p}\widetilde{\mathcal M}_{i,p}\to 0$, taking eventually a larger value of $p$, we have for every $r>1$
\begin{align*}
 r^{M-1} \left|\widetilde v_{i,p} (r)\right|^p & \le  r^{M-1} \left(V_M(\sqrt{\e}r)\right)^p =  r^{M-1}\left(1+\frac{\e r^2}{M(M-2)} \right)^{-\frac{M-2}{2}p} \\
& \le C r^{M-1-(M-2)p}
\end{align*}
which belong to $L^1(1,\infty)$ for $p>\frac{M}{M-2}$. For $r\in (0,1)$, instead  we have 
\[ r^{M-1} \left(\widetilde v_{i,p} (r)\right)^p \le 1 \]
by construction.
Then it is easy to see that the limit function $\widetilde v$ is a weak solution to the equation in \eqref{equazione-limite-v}.

Eventually one can see that  the limit function $\widetilde v$ is not null and satisfies $\widetilde v(0)=1$, and so that it coincides with the function $V_M$ identified by \eqref{soluzione-limite-v}, see also the Appendix. This can be seen by using the same arguments of \cite[Lemma 6]{Iaco}.
Indeed $s_{i,p}\widetilde {\mathcal M}_{i,p}$ is a critical point for $\widetilde v_{i,p}$ and integrating \eqref{eq:riscalate} on the interval between $s_{i,p}\widetilde {\mathcal M}_{i,p}$ and $t$ gives
\begin{equation}
\label{hatv'}
\widetilde v_{i,p}'(t)= - t^{1-M} \int\limits_{s_{i,p}\widetilde {\mathcal M}_{i,p}}^t r^{M-1} \widetilde v_{i,p}^{p} dr \; \text{ as } t_{i,p}\,\widetilde {\mathcal M}_{i,p} < t <  t_{i+1,p}\,\widetilde {\mathcal M}_{i,p} .
\end{equation}

Moreover for every $r>0$ \eqref{hp-b} assures that $s_{i,p}\widetilde {\mathcal M}_{i,p} <  r $ for $p$ near $p_M$ and so \eqref{hatv'} gives
\begin{align*} 
\widetilde v_{i,p}'(r) & = - r^{1-M} \int\limits_{s_{i,p}\widetilde {\mathcal M}_{i,p}}^r t^{M-1} \widetilde v_{i,p}^{p} dt 
\ge  - r^{1-M} \int\limits_{s_{i,p}\widetilde {\mathcal M}_{i,p}}^r t^{M-1}  dt \\
= & - \frac{r}{M} \left( 1- \left(\frac{s_{i,p}\widetilde {\mathcal M}_{i,p}}{r}\right)^M\right) \ge -  \frac{r}{M} .
\end{align*}
After, recalling that $\widetilde v_{i,p}(s_{i,p}\widetilde {\mathcal M}_{i,p} ) = 1$, we have
\[ \widetilde v_{i,p}(r)  = 1+ \int\limits_{s_{i,p}\widetilde {\mathcal M}_{i,p}}^r \widetilde v_{i,p}'(t) dt \ge 1 -  \int\limits_{s_{i,p}\widetilde {\mathcal M}_{i,p}}^r \frac{t}{M} dt = 1 - \frac{r^2}{2M} + \frac{(s_{i,p}\widetilde {\mathcal M}_{i,p})^2}{2M} .\]
Therefore by the pointwise convergence, and using \eqref{hp-b} once more, we get
\[ 1 \ge \widetilde v(r) \ge 1 - \frac{r^2}{2M} \] 
and the claim follows.

Since $\widetilde v$ is a weak solution to \eqref{equazione-limite-v} that satisfies $\widetilde v(0)=1$ then $\widetilde v=V_M$.
Let us also remark that we have proved that any sequence $p_n\to p_M$ admits a subsequence $p_{k_n}\to p_M$ for which $\widetilde v_{i,p_{k_n}}\to V_M$, which yields that $\widetilde v_{i,p}\to V_M$ indeed. 

Further  $\widetilde v_{i,p}\to V_M$ also in $C^1(R^{-1}, R)$ for every $R>1$. Indeed \eqref{hp-a} and \eqref{hp-c} ensure that  $t_{i,p}\,\widetilde {\mathcal M}_{i,p} < R^{-1}<R <  t_{i+1,p}\,\widetilde {\mathcal M}_{i,p}$ for $p$ near $p_M$. Therefore, remembering that $0\le \widetilde v_{i,p}\le 1$, we have by \eqref{hatv'}
\begin{align*}
|\widetilde v_{i,p}'(t)| \le t^{1-M} \left|\int\limits_{s_{i,p}\widetilde {\mathcal M}_{i,p}}^t r^{M-1} dr \right| \le C \; \text{ in } (R^{-1},R)
\end{align*}
thanks to \eqref{hp-b}.
Lastly it is easy to get an uniform bound for $\widetilde v_{i,p}''$ using the fact that $\widetilde v_{i,p}$ is a classical solution to  \eqref{eq:riscalate}  in $(R^{-1},R)$. 
}

\

It remains to prove that \eqref{hp-a}--\eqref{hp-c} hold true. To this aim we insert for a while the index denoting the number of nodal zones  and we let then $v_p^j$ be the nodal solution with $j$ nodal domains.	
By \eqref{def-w} we have that $w_{i,p}:=(t_{i,p}^m)^{\frac2{p-1}}v_p^m(t_{i,p}^m t)$ coincides with $v_p^i$ on $(0,1)$. This implies that 
\[{\mathcal M}^i_{i-1,p}=(t_{i,p}^m)^{\frac2{p-1}} {\mathcal M}^m_{i-1,p}\]
and also that 
\[\frac{s_{i-1}^m}{t_{i,p}^m}=s_{i-1,p}^i\]
which together yields
	\begin{align*}
	t^m_{i,p}\, \widetilde {\mathcal M}^m_{i-1,p}=\widetilde {\mathcal M}^i_{i-1,p} , 
	\qquad 
	s^m_{i,p}\, \widetilde {\mathcal M}^m_{i,p} =s^{i+1}_{i,p}\, \widetilde {\mathcal M}^{i+1}_{i,p} ,
	\end{align*}
	for $i=1,\dots m-1$.
	Therefore \eqref{massimi-pcritico-v}  implies \eqref{hp-a}.
	We claim that 
	\begin{equation}\label{hp-b-m}
	s^m_{m-1,p}\,\widetilde {\mathcal M}^m_{m-1,p}\to 0 \quad \text{ as } p\to p_M ,
	\end{equation}
	from which it follows \eqref{hp-b} and then, in turn, \eqref{hp-c}.

	For simplicity of notations we write $v_p$, $s_p$, $t_p$ and $\widetilde {\mathcal M}_p$ instead of  $v_p^m$, $s_{m-1,p}^m$, $t_{m-1,p}^m$ and $\widetilde M_{m-1,p}^m$.
	We begin by checking that
	\begin{equation}
	\label{bounded} s_p\,\widetilde {\mathcal M}_{p} \le C .
	\end{equation}
	We assume by contradiction that $s_p\,\widetilde {\mathcal M}_{p}\to +\infty$
	and look separately to the two cases 
	\begin{enumerate}[i)]	
		\item $\widetilde {\mathcal M}_{p} (t_p-s_p) \to 0$,	
		\item $\widetilde {\mathcal M}_{p} (t_p-s_p) \to   A \in [-\infty, 0)$. \end{enumerate}
	In the first case we look at the function $\widetilde{v}_p:=\widetilde v_{m-1,p}$ introduced in \eqref{v-tilde}. It is easy to see that $\widetilde v_{p}$ is positive, increasing and concave on $(a_p,b_p):=\left(t_p\widetilde {\mathcal M}_p,s_p\widetilde {\mathcal M}_p\right)$
	with $\widetilde{v}_p(a_p)=0< \widetilde{v}_p(t)< \widetilde{v}_p(b_p)=1$.
	So there exists a sequence $\xi_p\in (a_p,b_p)$ such that 
	\[ \widetilde v_p'(\xi_p) = \frac{\widetilde v_{p}(b_p) - \widetilde v_p(a_p) }{b_p-a_p} =\frac{1}{b_p-a_p} \to +\infty ,\] and by concavity also $\widetilde v_p'(a_p) \to +\infty$. 
	On the contrary the estimate \eqref{alternativa-lemma-3.10} yields
	\[ \widetilde v_p'(a_p) = \dfrac{1}{ {\mathcal M}_{m-1,p}^{\frac{p+1}{2}}} v_p'\left(\frac{a_p}{\widetilde {\mathcal M}_{m-1,p}}\right) \le \dfrac{C_2t_p^{1-p\frac{M-2}2}}{(\widetilde {\mathcal M}_{m-1,p})^{\frac{p+1}{p-1}}} = \dfrac{C_2t_p^{\frac {p+1}{p-1}+1-p\frac{M-2}2}}{(t_p \widetilde {\mathcal M}_{m-1,p})^{\frac{p+1}{p-1}}} \to 0   \] 
	because necessarily $t_p\widetilde {\mathcal M}_{m-1,p}$ diverges,  since we are assuming $i)$, while $\frac {p+1}{p-1}+1-p\frac{M-2}2$ is positive and converges to $0$.
	\\
	In the second case we introduce the notations
	\begin{align*}
	A_p  = (t_p- s_p)\,\widetilde {\mathcal M}_p , \qquad B_p  = \widetilde {\mathcal M}_p (1-s_p) , \\
	\hat v_{p}(t) = \frac{(-1)^{m-1}}{{\mathcal M}_{p}} v_p\left(\frac{t}{\widetilde {\mathcal M}_{p}}+s_{p}\right) \quad \text{ as } t\in [ A_p, B_p].
	\end{align*}
	Notice that $\hat v_p$ solves 
	\begin{equation}\label{hat-v-eq}\begin{cases}
	-\hat v_p'' - \frac{M-1}{t+\widetilde {\mathcal M}_{p}s_{p}} \hat v_p' = |\hat v_p|^{p-1} \hat v_p   &   t\in (A_p, B_p) , \\
	0<\hat v_p(t)\le \hat v_p(0)=1, \;  \hat v_p' (0) = 0 &   t\in (A_p, B_p) , \\
	\hat v_p(A_p) = 0=\hat v_p(B_p) , & 
	\end{cases}
	\end{equation}
	with $A_p \to  A  <0$  by assumption ii) and  
	$B_p \to +\infty$ by \eqref{massimi-pcritico-v}, \eqref{zeri-pcritico-v}.
	Integrating the equation in \eqref{hat-v-eq} we get for $t\in [0,B_p]$
	\begin{align*}
	\frac{|\hat v_p'(t)|}{t+\widetilde {\mathcal M}_{p}s_{p}} & = \dfrac{1}{(t+\widetilde {\mathcal M}_{p}s_{p})^M}
	\int_0^t (\tau+\widetilde {\mathcal M}_{p}s_{p})^{M-1} \hat v_p^{p}(\tau) \, d\tau 
	\\ 
	&\le \frac{1}{M}\left(1-\left(\frac{\widetilde{\mathcal M}_{p}s_{p}}{t+ \widetilde{\mathcal M}_{p}s_{p}}\right)^M\right) \le \frac{1}{M} .
	\end{align*}
	Besides taking $t\in [-\delta, 0]$ with $0<\delta<-A/2$ and integrating the equation in \eqref{hat-v-eq} on $(t,0)$ gives 
	\begin{align*}
	\frac{|\hat v_p'(t)|}{t+\widetilde {\mathcal M}_{p}s_{p}} & = \dfrac{1}{(t+\widetilde {\mathcal M}_{p}s_{p})^M}
	\int_t^0 (\tau+\widetilde {\mathcal M}_{p}s_{p})^{M-1} \hat v_p^p(\tau)\, d\tau  
	\\ 
	&\le \frac{1}{M}\left(\left(\frac{\widetilde{\mathcal M}_{p}s_{p}}{t+ \widetilde{\mathcal M}_{p}s_{p}}\right)^M-1\right) \le  \frac{1}{M}\left(\left(\frac{\widetilde{\mathcal M}_{p}s_{p}}{-\delta + \widetilde{\mathcal M}_{p}s_{p}}\right)^M-1\right) \le C(\delta)   .
	\end{align*}
	So $\hat v_p$ converges in $C^1[0, +\infty)$ to a bounded  weak solution of 
	\[-\hat v'' = \hat v^{p_M} \]
	which is non-trivial because $\hat v(0)=1$. This is not possible because $\hat v$ should be strictly convex.

	Now that it has been assured that $ s_p\,\widetilde {\mathcal M}_p $ is at least bounded, we take that \eqref{hp-b-m} does not hold, which means that  (up to a subsequence) $s_p\,\widetilde {\mathcal M}_p\to s_0>0$.
	We check that it is not possible by arguing separately according if
	\begin{enumerate}[I)]
		\item $t_p\,\widetilde {\mathcal M}_p\to  s_0$,
		\item $t_p\,\widetilde {\mathcal M}_p\to 0$,
		\item $t_p\,\widetilde {\mathcal M}_p\to  t_0 \in (0, s_0)$.
	\end{enumerate}	
	Case I) can be ruled out arguing as in the previous case i). Also here we get that
	$\widetilde v_p'(t_p\, \widetilde {\mathcal M}_p) \to +\infty$, while estimate \eqref{alternativa-lemma-3.10} would imply that it stays bounded.
	\\
	Otherwise in case II)   we consider again the function $\widetilde v_p:=\widetilde v_{m-1,p}$ introduced in \eqref{v-tilde}  and extended to zero outside  $(t_p \widetilde {\mathcal M}_p, \widetilde {\mathcal M}_p)$ so that it belongs to $\mathcal D_M(0,\infty)$ and by Corollary \ref{corollario} is uniformly bounded in $\mathcal D_M(0,\infty)$.
         Now $(t_p \widetilde {\mathcal M}_p, \widetilde {\mathcal M}_p)$ invades $(0,\infty)$ because we are taking that $t_p \widetilde {\mathcal M}_p\to 0$ and \eqref{massimi-pcritico-v} holds. Then the same arguments used in the first part of the proof show that $\widetilde v_p\to \widetilde v$ weakly in $\mathcal D_M(0,\infty)$  and in $C^1_{\loc}(0,\infty)$, where $\widetilde v$
	 weakly	solves	
		\begin{equation}\label{eq:critico}
		-\left(t^{M-1} \tilde{v}'\right)' = t^{M-1} \tilde{v}^{p_M}, \ \ \mbox{ as } t>0. 
		\end{equation}
		Therefore $\tilde v$ has to be a suitable rescaling of the function $V_M$, as showed in the Appendix. In particular it has only one critical point at $r=0$.
		On the other hand the functions $\widetilde v_p$ have a critical point at $s_p\widetilde{\mathcal M}_p\to s_0 >0$, and by the convergence in $C^1_{\loc}(0,\infty)$ $s_0$ is a critical point for $\tilde v$.
	\\
	 At last case III) can be ruled out following the line of case b) in the proof of Proposition \ref{conv-altre-zone-nodali-v}. Precisely we look at the function $w_p=w_{m-1,p}$ introduced in \eqref{def-w}, and check the hypotheses of Lemma \ref{lemma-w'(1)=0}. \eqref{hp-est}, i.e. $\rho_p= t^m_{m-1,p}\widetilde{\mathcal M}^m_{m-2,p}\to +\infty$ is ensured by \eqref{hp-a}. Concerning \eqref{hp-int}, it is trivial for $m=2$, while for $m\ge 3$ rescaling we get 
	\[ s^m_{m-2,p}\widetilde{\mathcal M}^m_{m-2,p} = s^{m-1}_{m-2,p}\widetilde{\mathcal M}^{m-1}_{m-2,p}\le C\] by the previously proved property \eqref{bounded}, so that \[\tau_p= s^m_{m-2,p}/t^m_{m-1,p} \le C /t^m_{m-1,p}\widetilde{\mathcal M}^m_{m-2,p}=C/\rho_p\to 0.\]
	So Lemma \ref{lemma-w'(1)=0} gives that $w_p\to 0$ uniformly on any set of type $[1-\d,1]$ with $0<\d<1$. In particular it is uniformly bounded on $[1-\d,1]$. 
	On the other hand $w_{p}$ is bounded also in $[1,1/t_p]$ (uniformly w.r.t.~$p$) because 
	\[  |w_p(r)| \le  t_{p}^{\frac{2}{p-1}} {\mathcal M}_p = 
	\left( t_{p}\widetilde {\mathcal M}_p \right)^{\frac{2}{p-1}} \le C  \]
by assumption.
Moreover $s_p/t_p$ is a critical point for $w_p$ which converges to $s_0/t_0$, and the corresponding maximum value is 
\[w_p(s_p/t_p)= t_p^{\frac{2}{p-1}}|v_p(s_p)| =\left( t_{p}\widetilde {\mathcal M}_p \right)^{\frac{2}{p-1}}\to t_0^{\frac{M-2}{2}}.\]  Integrating the equation in \eqref{eq-w} gives
\begin{align*}
|w'_p(r)| \le  r^{1-M} \int_{\frac{s_p}{t_p}}^r t^{M-1} |w_p(t)|^p dt \le C
\end{align*}
whenever $r\in (1-\d, R)$ for any fixed $R>1$. Next since $w_p$ is a classical solution to \eqref{eq-w} it is easily seen that also $|w''_p(r)|$ is bounded for $r\in (1-\d,R)$, so that $w_{p}$ converges in  $C^1_{\loc}(1-\d ,+\infty)$ to a function $w$ that weakly satisfies 
	\[\begin{cases}
	-\left( t^{M-1} w' \right)' =  t^{M-1} |w|^{p_M-1} w  \quad & \text{ as } t>1-\d , \\
	w(s_0/t_0)= t_0^{\frac{M-2}{2}} >0 , &  \\
	w(1)= 0 . &
	\end{cases}\]
	This is not possible because $w$ should be  identically zero by the unique continuation principle, as we have seen that $w$ coincides with zero on $(1-\d,1]$.
	\end{proof}

 \subsection{Some consequences of the convergence result} 
 We conclude this section by pointing out some qualitative properties of some auxiliary functions
\begin{align}
	\label{z_p}
z_p(r) & =rv_p'(r)+\dfrac{2}{p-1}v_p(r) & \text{ as } 0\leq r < 1,  & \\
\label{f_p}
f_p(r) & = p r^2|v_p|^{p-1}(r) & \text{ as }  0\leq  r < 1, \\
\label{tilde-f_p}
\tilde f_{i,p}(r) &= f_p \left(\frac{r}{\widetilde {\mathcal M}_{i,p}}\right)=p r^2 |\tilde v_{i,p}(r)|^{p-1} & \text{ as }  t_{i,p}\widetilde{\mathcal M}_{i,p}  < r < t_{i+1,p} \widetilde{\mathcal  M}_{i,p} ,
\end{align}
(for $i=0,\dots,m-1$)
that can be deduced by the convergence established in Propositions \ref{conv-altre-zone-nodali-v}, \ref{conv-sol-v} and  \ref{prop:conv-1-zona-v}, and shall be useful when investigating the asymptotic behavior of the eigenfunctions and eigenvalues related to $v_p$, in next section.

\begin{lemma}\label{lem:f-p}
	The function $z_p$ has exactly $m$ zeros in $(0,1)$, one in each nodal domain $(t_{i,p},t_{i+1,p})$ of $v_p$, that we denote by $\xi_{i,p}$ as $i=0,1,\dots m-1$.
	\\
	Moreover $\xi_{i,p}$ is the  unique critical point  in the nodal domain $(t_{i,p},t_{i+1,p})$ of the function $f_p$, which is strictly increasing in $(t_{i,p}, \xi_{i,p})$ and strictly decreasing in $(\xi_{i,p}, t_{i+1,p})$. 
\\
	Further $s_{i,p}<\xi_{i,p} <t_{i+1,p}$. 
\end{lemma}
Here we meant $t_{0,p}=0$ and $t_{m,p}=1$.
\begin{proof}
	The first part of the statement, concerning $z_p$, has been proved in \cite[Lemma 5.15]{AG-sez2}.
	Next it suffices to compute 	
	\[f'_p=(p-1)r|v_p|^{p-3}v_p\Big(\frac 2{p-1}v_p+r{v_p}'\Big)=(p-1)r|v_p|^{p-3}v_p z_p ,\]
	as $r\neq t_{i,p}$, 
and the second part of the statement follows trivially.
In particular $\xi_{i,p} >s_{i,p}$ because in the subset $(t_{i,p},s_{i,p})$ the functions $v_p$ and $v'_p$ have the same sign, so that $f_p'>0$.
\end{proof}

\begin{lemma}\label{lem:tilde-f_p}
For every $i=0,\dots m-1$, as $p\to p_M$ we have
\begin{align}
\label{tilde-f_p-lim}
& \tilde f_{i,p}(r) \to F(r)=\frac{(M+2)\, r^2}{M-2} \left(1+\frac{r^2}{M(M-2)}\right)^{-2} 
 \intertext{uniformly in $[R^{-1},R]$ for every $R>1$ and also in $[0,R]$ when $i=0$. Moreover }
\label{xi_p-lim}
& \xi_{i,p} \widetilde{\mathcal M}_{i,p}  \to \bar \xi \in(0,\infty)
\end{align}
where $\bar \xi$ is the unique maximum point of the function $F$.
\end{lemma}
\begin{proof}
	The convergence of $\tilde f_{i,p}$ is an immediate consequence of the one of $\widetilde v_{i,p}$ stated in Propositions \ref{prop:conv-1-zona-v} and  \ref{conv-sol-v}.  Notice that while proving Proposition  \ref{conv-sol-v} we have shown that $t_{i,p}\widetilde{\mathcal  M}_{i,p} \to 0$ and $t_{i+1,p} \widetilde{\mathcal  M}_{i,p} \to +\infty$. 
	Since the function $F$ has only one critical point $\bar \xi \in(0,+\infty)$, which is its maximum point, it follows that the maximum point of $\tilde f_{i,p}$ converges to $\bar \xi$. On the other hand it is clear by construction that the maximum point of $\tilde f_{i,p}$ is $\xi_{i,p} \widetilde{\mathcal M}_{i,p}$.
\end{proof}

Let us also recall an estimate obtained in \cite[Proposition 3.6]{DIPN>3}  for integer values of $M$ that we extend to every value of $M$.
\begin{lemma}\label{stima-f_p-1}
	The function $f_p$ satisfies $0\le f_p(r)\leq C$  as $r\in[0,1]$, uniformly w.r.t. $p$ in a left neighborhood of $p_M$.
\end{lemma}
We report here a slightly different proof, in view of further estimates that we aim to obtain.  
\begin{proof}
 The first  assertion of Lemma \ref{prop-3.6-DIPN>3} implies that for every  $r\in [0,t_{1,p})$ 
	\[0\leq f_p(r)\leq p \,  g_p(\widetilde {\mathcal M}_{0,p} r ) \quad \text{ being } \ g_p(s):=\frac{s^2}{(1+s^2)^{\frac{(M-2)(p-1)}2}}. \]
	Since the function $g_p$ are uniformly bounded on $[0,+\infty)$ (as $p\ge\frac{M}{M-2}$), it follows that also $f_p$ are uniformly bounded on $[0,t_{1,p}]$.	
\\
	Next we know that, for every $i=1,\dots,m-1$ and  $K>0$, $\widetilde v_{i,p}\to V_M$ uniformly in $[\frac 1K,K]$.  As $V_M$ has a positive minimum on the set $[\frac 1K,K]$, it follows that
	\[ |\widetilde {v}_{i,p}(r)  | \leq  2 \,V_M(r) \  \ \ \text{ in } [\frac 1K,K]\]
 as $p_M-\delta<p<p_M$ for some $\delta=\delta(K)>0$.

As in the previous step it follows that 
	\begin{equation}\label{numero}f_p(r)\leq  2 p \,  g_{p_M}(\widetilde {\mathcal M}_{i,p} r ) \leq C\end{equation} 	
	in the interval $[(K \widetilde {\mathcal M}_{i,p})^{-1}, K \widetilde {\mathcal M}_{i,p}^{-1}]$ {for $p\in(p_M-\delta,p_M)$}. \\
On the other hand in force of \eqref{xi_p-lim} we can choose the parameter $K$ in such a way that the maximum point of $f_p(r)$ in the interval $(t_{i,p},t_{i+1,p})$, i.e. $\xi_{i,p}$, is contained in $[(K \widetilde {\mathcal M}_{i,p})^{-1}, K \widetilde {\mathcal M}_{i,p}^{-1}]$, 
implying that
$0\le f_p(r)\leq C$ in the interval $(t_{i,p}, t_{i+1,p})$ for $i=1,\dots,m-1$ concluding the proof.
\end{proof}

Similar arguments allow also to show the following estimate.

\begin{lemma}\label{lemma-2.5}
	For every $\e>0$ there exist  $\bar K=\bar K(\e)>0$ and  $\bar p=\bar p(\e,\bar K)>0$ such that, denoting by 
	\begin{align*} 
	&  G_{i,p}(K):=\{r\in (0,1) : K (\widetilde {\mathcal M}_{ i-1,p})^{-1}<r<   (K\widetilde {\mathcal M}_{i,p})^{-1} \}  & \; \text{ as } i=1,\dots m-1,    \\
	&  G_{m,p}(K):=\{r\in (0,1) : K (\widetilde {\mathcal M}_{m-1,p})^{-1}<r< 1  \}    & 
	\end{align*}
	it holds
	\begin{equation}\label{piccolo}
	\max\left\{ f_p(r) :  r\in  \bigcup_{i=1}^{m} G_{i,p}(K)  \right\}<\e
      \end{equation}
      for any $K>\bar K$ 	provided that $p\in (\bar p, p_M)$.
\end{lemma}
\begin{proof}
  To begin with we choose  $\bar K>0$ such that $ K>\max\{\bar{\xi}, \bar{\xi}^{-1}\}$ and  $ p_M g_{p_M}( K^{-1})$, $p_M g_{p_M}( K)$ $< $ $\e/2 $  for any $K>\bar K$. Here $\bar \xi$ is the maximum point of the function $F$ mentioned in Lemma \ref{lem:tilde-f_p} and $g_{p_M}$ is the same function introduced in the proof of Lemma \ref{stima-f_p-1}, and the choice of $\bar K$ is possible because $g_{p_M}(0)=0=\lim\limits_{r\to+\infty} g_{p_M}(r)$. 
\\ 
	 Next \eqref{numero} yields that there exists $p_1=p_1(\bar K,\e)$ such that 
	 	\[ f_p((\bar K\widetilde {\mathcal M}_{i,p})^{-1}), \, f_p(\bar K\widetilde {\mathcal M}_{i,p}^{-1})	 <\e \quad \text{ as $p_1<p<p_M$ and $i=0,\dots m-1$}.\]
                Then \eqref{xi_p-lim} yields that there exists  $\bar p=\bar p(\bar K)>p_1$ such that $\xi_{i,p}$, the unique critical point of $f_p$ in the interval $( t_{i,p}, t_{i+1,p})$, satisfies 
	 \[ (K  \widetilde{\mathcal M}_{i,p})^{-1}< (\bar K  \widetilde{\mathcal M}_{i,p})^{-1}  <{\xi}_{i,p} <\bar K  \widetilde{\mathcal M}_{i,p}^{-1} <K  \widetilde{\mathcal M}_{i,p}^{-1}\; \quad \text{ as  $\bar p<p<p_M$} \]
 and $i=0,\dots m-1$, for any $K>\bar K$.	  Remembering also that $f_p$ is increasing in $(t_{i,p},\xi_{i,p})$ and decreasing in $(\xi_{i,p},t_{i+1,p})$ by Lemma \ref{lem:f-p}, it follows that 
	 \[\begin{array}{lll}
	  f_p(r)\le f_p\left((K \widetilde{\mathcal M}_{i,p})^{-1}\right) <\e  &  \text{ for  }K (\widetilde {\mathcal M}_{i,p})^{-1}<r<t_{i+1,p}, \  & \text{ as } i=0,\dots, m-1, \\
	  f_p(r)\le f_p\left(K \widetilde{\mathcal M}_{i,p}^{-1}\right) <\e & \text{ for  } t_{i,p}<r<(K\widetilde {\mathcal M}_{i,p})^{-1}, \  & \text{ as } i=1,\dots, m-1 
	 \end{array}\]
	 for any $K>\bar K$, for the same values of $p$.
\end{proof}

\section{The computation of the Morse index}\label{se:3}

In this section we address to the computation of the Morse index of the nodal radial solution $u_p$ of \eqref{H} when $p$ approaches the threshold $p_{\alpha}$.
	By definition  the Morse index of $u_p$,  that we denote by $m(u_p)$, is  the maximal dimension of a subspace of $H^1_0(B)$ in which the quadratic form 
	\[{\mathcal Q}_p(w):=\int_B\left(|\nabla w|^2 -p\,|x|^{\alpha}|u_p|^{p-1}w^2\right) dx
		\]
                is negative defined, or equivalently,
                is the number, counted with multiplicity, of the negative eigenvalues in $H^1_0(B)$ of  
\begin{equation}\label{eigenvalue-problem}
\left\{\begin{array}{ll}
-\Delta  \phi-p|x|^\a|u_p|^{p-1}\phi =\L_i(p)\, \phi & \text{ in } B\\
\phi= 0 & \text{ on } \partial B.
\end{array} \right.
\end{equation}
Similarly the radial Morse index of $u_p$, denoted by $m_\rad(u_p)$, is  the number  of negative eigenvalues of  \eqref{eigenvalue-problem} in $H^1_{0,\rad}(B)$, namely the eigenvalues of \eqref{eigenvalue-problem} associated with a radial eigenfunction.
 It has been proved in \cite[Propositions 3.4, 4.1]{AG-sez2} (since $p|x|^\a|u_p|^{p-1}\in L^{\infty}(B)$) that
the number of negative eigenvalues  of \eqref{eigenvalue-problem}  in $H^1_0(B)$ (or in $H^1_{0,\rad}(B)$), counted with multiplicity, coincides with the number 
of negative eigenvalues of the singular eigenvalue problem 
\begin{equation}\label{singular-eigenvalue-problem}
\left\{\begin{array}{ll}
-\Delta  \widehat\phi-p|x|^\a|u_p|^{p-1}\widehat\phi =\dfrac{\widehat\L_i(p)}{|x|^2}\widehat\phi & \text{ in } B\setminus\{0\}\\
\widehat\phi= 0 & \text{ on } \partial B,
\end{array} \right.
\end{equation}
 in $H^1_0(B)$ (or in $H^1_{0,\rad}(B)$). This allows to give this alternative definition of Morse index:
  \begin{definition}[Alternative definition of Morse index]\label{def:Morse-alternativo}
    The Morse index of $u_p$ is the number, counted with multiplicity of the negative singular eigenvalues $\widehat\L_i(p)$ of \eqref{singular-eigenvalue-problem} in $H^1_0(B)$. Moreover the radial Morse index of  $u_p$ is the number of negative singular radial eigenvalues $\widehat\L_i^{\rad}(p)$ of \eqref{singular-eigenvalue-problem} in $H^1_{0,\rad}(B)$.
    \end{definition}
  These eigenvalues $\widehat\L_i(p)$ are well defined in $H^1_0(B)$  (by the Hardy inequality) as far as $\widehat\L_i(p)<\left(\frac{N-2}2\right)^2$
 and have the useful property that can be decomposed as
\begin{equation}\label{decomposition}
\widehat \L _i(p)=\widehat \L_k^{\rad}(p)+\l_j,
\end{equation}
where $\l_j=j(N+j-2)$ are the eigenvalues of the Laplace-Beltrami operator on the sphere ${\mathbb S}_{N-1}$, and $\widehat \L_k^{\rad}(p)$ are the radial singular eigenvalues of \eqref{singular-eigenvalue-problem} which are all simple, see \cite{AG-sez2} where a complete study of the singular eigenvalues and their properties has been done. 
Further if $\widehat \phi$ is a radial eigenfunction of \eqref{singular-eigenvalue-problem}, the function
\[\psi(t)=\widehat \phi(r) \ \text{ with } \ \ t=r^{\frac{2+\a}2}\]
is a generalized radial singular eigenfunction of 
the singular Sturm-Liouville problem 
	\begin{equation}\label{radial-general-H-no-c}
\left\{\begin{array}{ll}
- \left(t^{M-1} \psi'\right)'- t^{M-1} p|v_p|^{p-1} \psi = t^{M-3} \widehat{\nu}_i(p)  \psi & \text{ for } t\in(0,1)\\
\psi\in  H^1 _{0,M}
\end{array} \right.
\end{equation} 
 where $v_p$ as in \eqref{transformation-henon-no-c} is a solution to \eqref{LE-radial} as in Section \ref{sec:da-u-a-v} and
$M=M(\alpha, N)$ has been defined in \eqref{M-a}. These eigenvalues $\widehat{\nu}_i(p)$ are well defined in $H^1_{0,M}$ as far as $\widehat{\nu}_i(p)<\left(\frac{M-2}2\right)^2$ and satisfy
\begin{equation}
\label{relazione-autov-no-c}
\widehat\L^{\rad}_{i}(p) = \left(\frac{2+\alpha}{2} \right)^2 \widehat{\nu}_i(p). 
\end{equation}
To deal with problem \eqref{radial-general-H-no-c} we define 
by ${\mathcal L}_M$ the  Lebesgue space
 \[{\mathcal L}_M:=\{w:(0,1)\to \R \text{ measurable and s.t. } \int_0^1 t^{M-3} w^2 dt < +\infty\}\]
  with the scalar product
  $\int_0^1r^{M-3}\psi w\ dr$, which gives the orthogonality condition
	\[ w\underline \perp_{M} \psi \Longleftrightarrow \int_0^1 t^{M-3} w \psi dt =0  \quad   \text{for } w,\psi\in \mathcal L_M. \]
	In virtue of an extended radial Hardy inequality for $H^1_{0,M}$ in  \cite[Lemma 6.5]{AG-sez2} $H^1_{0,M}  \subset {\mathcal L}_M$ and this allows to 
characterize the eigenvalues $\widehat{\nu} $ by the minimization problems 
\begin{equation}\label{nu-var}
\begin{split}
\widehat{\nu}_1(p)=\inf_{\substack{w\in {H}^1_{0,M}\\ w\neq 0}}\frac{\int_0^1 t^{M-1}\left((w')^2 -p|v_p|^{p-1}w^2 dt\right) dr }{\int_0^1 t^{M-3}w^2 dt},  
\\
\widehat{\nu}_{i}(p)=\inf_{\substack{w\in   {H}^1_{0,M} \\ w\neq 0\\ w\underline \perp_{M}\{\psi_1,\dots,\psi_{i-1}\}}}\frac{\int_0^1 t^{M-1}\left((w')^2 -p|v_p|^{p-1}w^2 dt\right) dr }{\int_0^1 t^{M-3}w^2 dr}  
\end{split}\end{equation}
where $\psi_j$ for $j=1,\dots,m-1$ denotes an eigenfunction associated with $\widehat{\nu}_{j}$. Every time $\widehat{\nu}_i<\left(\frac{M-2}2\right)^2$, the function which attains $\widehat{\nu}_i$ is a weak solution to \eqref{radial-general-H-no-c} meaning that
\begin{equation}\label{radial-general-weak-H-no-c}
\int_0^1 t^{M-1} \psi'\varphi' \ dt- p\int_0^1 t^{M-1} |v_p|^{p-1} \psi\varphi\ dt=\widehat{\nu}_i(p)\int_0^1  t^{M-3}   \psi\varphi \ dt
\end{equation}
for every $\varphi \in H^1_{0,M}$.
These generalized radial singular eigenvalues $\widehat{\nu}_i(p)$, (associated with $v_p$) have been studied in \cite[Subsection 3.1]{AG-sez2} where it is proved that they are all {\em simple}, eigenfunctions associated with different eigenvalues are orthogonal in $\mathcal L_M$,  the only negative eigenvalues of \eqref{radial-general-H-no-c} are
\begin{align}
  \label{nu_i-neg} &  \widehat \nu_1(p)<\widehat \nu_2(p)<\dots<\widehat \nu_m(p)<0 &  \\
\intertext{and satisfy}
\label{nl<k-general-H} & \widehat{\nu}_i(p)  < -(M-1)   &\text{ as } i=1,\dots m-1 ,
\\
\label{num>k-general-H} & -(M-1) <\widehat{\nu}_m(p) <0  ,  &
\end{align}
for any value of the parameter $p$. Then \eqref{relazione-autov-no-c}, together with Definition \ref{def:Morse-alternativo}, implies that $m_{\rad}(u_p)=m$, the number of the nodal zones of $u_p$.\\
Furthermore putting together Proposition 1.5 and Theorem 1.7 from \cite{AG-sez2}  we have 
\begin{proposition}\label{general-morse-formula-H} Let $\a\geq 0$ and let $u_p$ be any radial solution to \eqref{H} with $m$ nodal zones.
The Morse index of $u_p$ is given by
	\begin{align}\label{tag-2-H}
	m(u_p) & =\sum\limits_{i=1}^{m}\sum\limits_{0\le j<J_i(p)} N_j, \intertext{where} \nonumber
	J_{i}(p) & =\frac{2+\a}{2} \left(\sqrt{\left(\frac{M-2}{2}\right)^2- \widehat{\nu}_i(p)}-\frac{M-2}{2}\right) \intertext{ and } \nonumber
	N_j &= \frac{(N+2j-2)(N+j-3)!}{(N-2)!j!}
	\end{align}
	stands for  the multiplicity of the eigenvalue  $\l_j=j(N+j-2)$ of the Laplace-Beltrami operator in the sphere ${\mathbb S}_{N-1}$.
\end{proposition}
Therefore the asymptotic Morse index of $u_p$   as $p\to p_\a$ can be deduced, by the asymptotic behavior of the generalized radial singular eigenvalues $\widehat{\nu}_i(p)$ and of the related eigenfunctions $\psi_{i,p}$ of \eqref{radial-general-H-no-c}  as $p\to p_M$ which are associated with the function $v_p$ defined in \eqref{transformation-henon-no-c} and studied in Section \ref{sec:da-u-a-v}. This 
will be the topic of the remaining of this section.

\subsection{Asymptotic of the singular eigenvalues { $\widehat{\nu}_i(p)$ for $i=1,\dots,m-1$}}\label{3:m-1}

\

For simplicity of notations in the present subsection and in the next one we shall write $\nu_j(p)$ instead of $\widehat\nu_j(p)$, and 
we will denote by $\psi_{j,p}\in H^1_{0,M}$ the corresponding eigenfunction  to \eqref{radial-general-H-no-c} normalized such that  
	\begin{align}\label{normalization}
	\int_0^1  r^{M-3} \psi_{j,p} \psi_{k,p} dr= \delta_{jk} .
	\end{align}	
For every $i=0,\dots m-1$ and $j=1,\dots,m$ we also introduce the rescaled eigenfunctions
\begin{equation}\label{rescaled-eigenf}
\widetilde \psi_{j,p}^i(r):=\begin{cases}
 (\widetilde{ \mathcal M}_{i,p})^{\frac{2-M}2}  \psi_{j,p}\Big(\frac r{ \widetilde {\mathcal M}_{i,p}} \Big) \quad & \text{ if } \ \widetilde {\mathcal M}_{i,p}t_{i,p}<r<\widetilde {\mathcal M}_{i,p}t_{i+1,p} ,
\\
0 & \text{ elsewhere,}
\end{cases}\end{equation}
where $t_{i,p}, t_{i+1,p}$ are the zeros of $v_p$ as in Section \ref{sec:da-u-a-v} and $\widetilde{ \mathcal M}_{i,p}$ is as in \eqref{m-tilde}, 
in such a way that 
\begin{align}\label{normalization-tilde}
\int_0^{\infty} r^{M-3} (\widetilde{\psi}_{j,p}^i)^2\, dr & =\int_0^1 r^{M-3}\psi_{j,p}^2\ dr=1 , \\ 
	\label{normalization-tilde-primo}
\int_0^{\infty} r^{M-1} ((\widetilde{\psi}_{j,p}^i)')^2\, dr & =\int_0^1 r^{M-1}(\psi_{j,p}')^2\ dr .
\end{align}
Then the functions $\widetilde{\psi}_{j,p}^i$  belong to the space $\mathcal D_M(0,\infty)$ for every $i=0,\dots, m-1$ and $j=1,\dots,m$ since $\psi_j\in H^1_{0,M}$ and they satisfy
\begin{equation}\label{psi-i-p-rescaled}
-\left(r^{M-1}{ (\widetilde \psi^i_{j,p}})'\right)'=r^{M-1}\left( { W^i_p}+\frac{\nu_j(p)}{r^2}\right) \widetilde \psi_{j,p}^i \quad
\text{ as } \widetilde {\mathcal M}_{i,p}t_{i,p}<r < \widetilde {\mathcal M}_{i,p}t_{i+1,p},
\end{equation}
where  
\begin{equation}\label{V_i}
{ W^i_p}(r)= p\, |\widetilde v_{i,p}(r)|^{p-1}\end{equation}
and $\widetilde v_{i,p}$ is as defined in \eqref{v-tilde}.
By the asymptotic of $\widetilde v_{i,p}$ in Propositions \ref{conv-sol-v} and \ref{prop:conv-1-zona-v} we have that 
\begin{equation}\label{uniform-convergence}
{ W^i_p}(r)\to { W}(r)= \frac{M+2}{M-2}\left(1+\frac {r^2}{M(M-2) } \right)^{-2}  \end{equation}
in $C^1_{\loc}[0,\infty)$ for $i=0$ and in $C^{1}_{\loc}(0,\infty)$ for $i=1,\dots,m-1$, therefore the eigenvalue problems \eqref{psi-i-p-rescaled} have a unique limit problem which is the following
\begin{align}\label{eq:finale}
-\left(r^{M-1}(\widetilde\psi)'\right)'= r^{M-1} \left( W  +\frac{ \beta}{r^2} \right) \widetilde\psi \ \text{ as }\ r\in(0,\infty),
\end{align} 	   
and admits as nonpositive eigenvalues in the space $\mathcal D_M(0,\infty)$ only the two values
$\beta_1=-(M-1)$ and $\beta_2=0$ with corresponding eigenfunctions 
\begin{equation}\label{eta}
\eta_1(r)=\frac r{\big(1+\frac {r^2}{M(M-2)}\big)^{\frac M2}} , \qquad 
\eta_2(r)= \frac{1-\frac{r^2}{M(M-2)}}{\big( 1+\frac {r^2}{M(M-2)}\big)^{\frac M2}}
\end{equation}
see the Appendix.
 We recall that an eigenfunction $\eta$ is a weak solution to \eqref{eq:finale} if it satisfies
\begin{equation}\label{eq:finale-weak}
\int_0^\infty r^{M-1}\eta '\varphi'\ dr=\int_0^\infty r^{M-1}\left( W  +\frac{ \beta}{r^2} \right) \eta \varphi\ dr\end{equation}
for every $\varphi\in \mathcal D_M(0,\infty)$.
\

Let us prove some useful lemmas, which inherit all the $m$ negative eigenvalues and the related eigenfunctions.
\begin{lemma}\label{lemma-stime-insieme}
	There exist $\delta >0$ and $C>0$ such that 	for every $p\in (p_M-\d,p_M)$ we have 
	\begin{align}\label{eq-stima-autov} 
		-C \le \nu_1(p) <  \nu_2(p) \dots <  \nu_{m}(p)   <0
\\ 
\label{eq-stima-norma}	
 \int_0^{\infty} r^{M-1} ((\widetilde\psi_{j,p}^i)')^2\, dr \le C 
	\end{align}
for every $i=0,\dots, m-1$ and $j=1,\dots m$.
\end{lemma}
\begin{proof}
Using $\psi_{j,p}$ as a test function in \eqref{radial-general-weak-H-no-c} gives
\begin{align}\label{line}
\begin{split}\int_0^1 r^{M-1}\left(\psi_{j,p}'\right)^2& =\int_0^1 r^{M-1} \left( p|v_p|^{p-1} +\frac{\nu_j(p)}{r^2} \right)\psi_{j,p}^2 dr \\ &= 
\int_0^1 r^{M-3} \left(f_p +\nu_j(p)\right) \psi_{j,p}^2 dr \end{split}
\end{align}
where $f_p$ is as defined in \eqref{f_p}.
Taking advantage from \eqref{normalization} one can extract $\nu_1(p)$  getting that
\[
\begin{split}
\nu_1(p) & =\int_0^1 r^{M-1} \left(\psi_{1,p}'\right)^2-r^{M-3}
f_p \, \psi_{1,p}^2\, dr\geq-\sup_{r\in (0,1)}f_p(r)
\int_0^1r^{M-3}\psi_{1,p}^2 \, dr= -C\end{split}\]
for $p$ near at $p_M$, thanks to Lemma \ref{stima-f_p-1}. 
\\
Besides, since $\nu_j(p)<0$  for $j=1,\dots,m$ by \eqref{nu_i-neg}, \eqref{line} also yields that
\[
\int_0^1 r^{M-1}\left(\psi_{j,p}'\right)^2
<  \int_0^1 r^{M-3}f_p\psi_{j,p}^2\, dr \le \sup_{r\in(0,1)}f_p(r) \int_0^1 r^{M-3}\psi_{j,p}^2\, dr = C .
\]
So also \eqref{eq-stima-norma} is proved, recalling \eqref{normalization-tilde-primo}.
\end{proof}	

From the boundedness of the eigenfunctions  in \eqref{eq-stima-norma} it is easy to deduce that they converge to eigenfunctions of the limit problem  \eqref{eq:finale}.

\begin{lemma}\label{autofunz-limite} 
Let $j=1,\dots m$ and $p_n$ a sequence in $(1,p_M)$ with $p_n\to p_M$. Then there exist an extracted sequence (that we still denote by $p_n$),  a number $\bar \nu_j \le 0$, 
a weak solution to \eqref{eq:finale} with $\beta=\bar\nu_j$, say it $\eta$, and $m$ numbers $A^0_j, \dots A^{m-1}_j\in \R$ such that
\begin{align*}
\nu_j(p_n) & \to  \bar \nu_j\\
\widetilde{\psi}_{j,p_n}^i& \to A^i_j\eta \quad \mbox{ weakly in $\mathcal D_{M}(0,\infty)$ and strongly  in $C^1_{\loc}(0,\infty)$}
\end{align*}
for $i=0,\dots,m-1$. 
\\
Further for $j=1,\dots m-1$ the function $\widetilde{\psi}_{j,p_n}^0$ converges to $ A^0_j\eta$ also in $C^1_{\loc}[0,\infty)$.
\end{lemma}
\begin{proof}
By \eqref{nl<k-general-H}, \eqref{num>k-general-H} and \eqref{eq-stima-autov} it is clear that there is an extracted sequence $\nu_{j,p_n}\to \bar\nu_j\le 0$.
Moreover the normalization \eqref{normalization-tilde} and the estimate \eqref{eq-stima-norma}   imply that $\widetilde\psi_{j,p}^{i}$ are uniformly bounded in 
$\mathcal D_M(0,\infty)$ for $i=0,\dots,m-1$.  Then, up to another extracted subsequence $\widetilde\psi_{j,p_n}^{i}$ converges to a function $\eta$ weakly in ${\mathcal D}_{M}(0,\infty)$.
It is not hard to see that one can pass to the limit in the weak formulation of \eqref{psi-i-p-rescaled}, getting that $\eta$ is a weak solution to \eqref{eq:finale} with $\beta=\bar\nu_j\leq 0$.
Indeed \eqref{hp-a} and \eqref{hp-c} ensure that, for every $\varphi \in C^\infty_0(0,\infty)$ and for $n$ sufficiently large, the support of $\varphi$ is contained in $( t_{i,p_n}\widetilde {\mathcal M}_{i,p_n}, t_{i+1,p_n}\widetilde {\mathcal M}_{i,p_n})$, where  \eqref{psi-i-p-rescaled} holds. Moreover 
$\widetilde \psi_{j,p_n}^{i}$ converges to $\eta$ also in  $L^2_M(R^{-1},R)$ as well as in $\mathcal L_M(R^{-1},R)$ for every $R>1$, by \cite[Lemma 6.4]{AG-sez2}.
\\
Besides  $\eta\in \mathcal D_{M}(0,\infty)$ and hence $\eta\in H^1_M(0,R)$ for every $R>0$, and by \cite[VIII.2]{Bbook} $\eta\in C^1(0,R)$.
If $r_1,r_2> R^{-1}>0$ we have
\begin{align*}
\left|\widetilde\psi_{j,p}^{i}(r_1)-\widetilde\psi_{j,p}^{i}(r_2)\right|  & \le \int_{r_1}^{r_2} |(\widetilde\psi_{j,p}^{i})'(t)| dt \underset{\text{Holder and \eqref{eq-stima-norma}}}{\le} 
C\left(\int_{r_1}^{r_2} t^{1-M} dt\right)^{\frac{1}{2}} \\
& \le C R^{\frac {M-1}2} \sqrt{|r_1-r_2|},
\end{align*}
so the Ascoli Theorem ensures that (up to another extracted sequence) $\widetilde{\psi}_{j,p_n}^i\to \eta$ uniformly in any set of type  $[R^{-1}, R]$.
Next taking advantage from the equation in \eqref{psi-i-p-rescaled} it is easy to get a bound for $\widetilde\psi_{j,p}^{i}$ in $C^2(R^{-1}, R)$ which ensures that it actually converges in $C^1(R^{-1}, R)$.

Further when $i=0$ we also know that $W^0_{p_n}$ is uniformly convergent (and therefore uniformly bounded) on any set of type $[0,R]$.
	Consequently the arguments in \cite[Lemma 5.9]{DGG} and \cite[Proposition 3.9]{AG-sez2} prove that 
	\begin{equation}
	\label{conv-in-0}
	\left| (\widetilde\psi_{j,p_n}^0)'(r)\right| \le C r^{\theta_j(p_n)-1} , \qquad \theta_j(p_n)= \sqrt{\left(\frac{M-2}{2}\right)^2 -\nu_j(p_n)} -\frac{M-2}{2}
	\end{equation}
on $[0,R]$. Moreover when $j=1,\dots m-1$ the estimate  \eqref{nl<k-general-H} ensures that $\theta_j(p_n) > 1$ for every $n$. Therefore \eqref{conv-in-0} states that $(\widetilde\psi_{j,p_n}^0)'$ is uniformly bounded also in $[0,R]$, and Ascoli Theorem gives uniform convergence of $\widetilde\psi_{j,p_n}^0$  in $[0,R]$ as before. The $C^1$ convergence then follows from the uniform converge of $\widetilde\psi_{j,p_n}^0$ recalling that integrating \eqref{psi-i-p-rescaled} and using \eqref{conv-in-0}
one easily gets
\[(\widetilde\psi_{j,p_n}^0)'=-r^{1-M}\int_0^r t^{M-1}\left( W_{p_n}^0+\frac {\nu_j(p_n)}{t^2}\right) \widetilde \psi_{j,p_n}^0\ dt.
\]

\end{proof}
	
\begin{remark}\label{remark-dim}
Since the eigenvalues and eigenfunctions of the limit problem 	\eqref{eq:finale} are known, an immediate consequence of  Lemma \ref{autofunz-limite} is that or $A^i_j=0$ for every $i=0,\dots,m-1$, or $\bar \nu_j$ takes one of the values  $-(M-1)$ and $0$, and $\eta=\eta_1$ (if $\bar \nu_j=-(M-1)$) or $\eta=\eta_2$ (if $\bar \nu_j=0$).
	Further when $j=1,\dots m-1$ the inequality \eqref{nl<k-general-H} ensures that $\bar \nu_j=-(M-1)$ and therefore $\eta= \eta_1$. 
	Concerning $j=m$, the corresponding inequality \eqref{num>k-general-H} leaves open also the possibility $\bar \nu_m =0$  and $\eta=\eta_2$.
\end{remark}

The previous remark put in evidence that the eigenvalue $\nu_m$ has to be treated separately. We deal by now with the first $m-1$ eigenvalues and show that 
\begin{proposition}\label{prop:limit-nu-j<m} 
Let $j=1,\dots m-1$, then
\[ \lim\limits_{p\to p_M} \nu_j(p)= -(M-1).\]
Moreover for any sequence $p_n$ in $(1,p_M)$ with $p_n\to p_M$ there exist an extracted sequence (that we still denote by $p_n$), and $m$ numbers $A^0_j, \dots A^{m-1}_j\in \R$ not simultaneously null such that
\[\widetilde{\psi}_{j,p_n}^i\to A^i_j\eta_1\] 
 for every $i=0,\dots,m-1$, weakly in ${\mathcal D}_{M}(0,\infty)$, and strongly in $C^1_{\loc}(0,\infty)$, and also in $C^1_{\loc}[0,\infty)$ for $i=0$. 
\end{proposition}

\begin{proof}
As mentioned in Remark \ref{remark-dim}, it suffices to rule out the possibility that for every $i=0,\dots,m-1$, 
\begin{equation}
\label{contradiction}
\widetilde{\psi}_{j,p}^i\to 0 \ \text{uniformly in any set $[R^{-1}, R]$ or $[0,R]$ as $i=0$.}
\end{equation}	
	We show here that if  \eqref{contradiction} holds true then
	\begin{equation}\label{contradiction-2}
\int_0^1 r^{M-3}f_p \psi_{j,p}^2 \, dr \to 0 ,
	\end{equation}
	where $f_p$ is as in \eqref{f_p}. This is not possible (and so the proof is completed) because repeating the computations in the proof of Lemma \ref{lemma-stime-insieme} gives 
\[\begin{split}
& -(M-1)>\nu_j(p)=\int_0^1r^{M-1}(\psi_{j,p}')^2\, dr-\int_0^1 r^{M-3}f_p(r) \psi_{j,p}^2 \, dr\geq -\int_0^1 r^{M-3}f_p(r) \psi_{j,p}^2 \, dr .
\end{split}
\]
To check \eqref{contradiction-2} we begin by taking any $\e>0$, applying Lemma \ref{lemma-2.5} and splitting the integral as
\begin{align*}
\int\limits_0^1 r^{M-3} f_p \psi_{j,p}^2 dr & =\int\limits_0^{K (\widetilde{ \mathcal M}_{0,p})^{-1}} r^{M-3} f_p \psi_{j,p}^2 dr  +\sum_{i=1}^{m-1}\int\limits_{(K\widetilde{ \mathcal M}_{i,p})^{-1}} ^{K(\widetilde{ \mathcal M}_{i,p})^{-1}}  r^{M-3} f_p  \psi_{j,p}^2 dr \\
&  +\sum_{i=1}^{m}\int_{G_{i,p}(K)}  r^{M-3} f_p \psi_{j,p}^2 dr 
 \end{align*}
where $K$ (and consequently $G_{i,p}(K)$)
is chosen in such a way to satisfy \eqref{piccolo}. So using also \eqref{normalization-tilde} we obtain
\begin{align*}   \sum_{i=1}^{m}\int_{G_{i,p}(K)}  r^{M-3} f_p \psi_{j,p}^2 dr  < \e  \int_0^1 r^{M-3} \psi_{j,p}^2 dr = \e .\end{align*} 
On the other hand exploiting the uniform convergence stated in \eqref{uniform-convergence} we also have
\begin{align*}
\int\limits_0^{K (\widetilde{ \mathcal M}_{0,p})^{-1}} r^{M-3} f_p \psi_{j,p}^2\, dr
& = p \int\limits_0^{K (\widetilde{ \mathcal M}_{0,p})^{-1}} r^{M-1} |v_{p}|^{p-1} \psi_{j,p}^2\, dr \\
& =  \int_0^K s^{M-1} W^0_p(\widetilde {\psi}_{j,p}^0)^2\, ds
\to \int_0^K s^{M-1} W (\widetilde {\psi}_{j}^0)^2\, ds=0
\intertext{ by \eqref{contradiction}, and similarly }
\int\limits_{(K\widetilde{ \mathcal M}_{i,p})^{-1}} ^{K(\widetilde{ \mathcal M}_{i,p})^{-1}}  r^{M-3} f_p  \psi_{j,p}^2\, dr
& = p \int\limits_{(K\widetilde{ \mathcal M}_{i,p})^{-1}} ^{K(\widetilde{ \mathcal M}_{i,p})^{-1}}  r^{M-1} |v_{p}|^{p-1} \psi_{j,p}^2\, dr \\
& =  \int_{K^{-1}}^K s^{M-1} W^i_p (\widetilde {\psi}_{j,p}^i)^2\, ds
\to \int_{K^{-1}}^K s^{M-1} W (\widetilde {\psi}_{j}^i)^2\, ds=0 .
\end{align*}
Summing up we have proved that $\limsup\limits_{p\to p_M} \int_0^1 r^{M-3}f_p \psi_{j,p}^2 \, dr < \e$ for every positive $\e$ which clearly gives \eqref{contradiction-2} since $f_p\ge 0$.

\end{proof}

\subsection{The last negative eigenvalue}\label{3:m}
As mentioned before, the last negative eigenvalue $\nu_m(p)$ has a different behavior from the first $m-1$ ones, which is enlightened by the different global bounds \eqref{nl<k-general-H} and \eqref{num>k-general-H}. In the case of Lane-Emden equation studied in \cite{DIPN>3}
the relation \eqref{num>k-general-H} is sufficient to determine its contribution to the Morse index, therefore there is no need for further investigation. This is not the case anymore for the H\'enon equation, where the exact computation of its limit is necessary to compute the asymptotic Morse index.

 To this aim a more detailed knowledge of the asymptotic behavior of the previous $m-1$ eigenfunctions may help.

\


\begin{lemma}\label{lem:int-va-zero}
	For every $\d>0$ there exist $\bar K >1$ and $\bar p \in (1,p_M)$ such that 
	\begin{equation}\label{pezzo-Gi-piccolo-0} 
	\int_{G_{i,p}(K)} r^{M-3}\psi_{j,p}^2 dr <\d \qquad \mbox{ as $i=1,\dots m$, $j=1,\dots m-1$,}
	\end{equation} 
	for every $K>\bar K$ and $p\in (\bar p, p_M)$. 
\end{lemma}
Here $G_{i,p}(K)=\left( K(\widetilde {\mathcal M}_{i-1,p})^{-1} , (K\widetilde {\mathcal M}_{i,p})^{-1} \right)$ denotes the subset of $(0,1)$ introduced in Lemma \ref{lemma-2.5}.
\begin{proof}
Let $\e \in (0,1/2)$. By Lemma \ref{lemma-2.5}  we can chose  $\bar K_1(\e)$ and $\bar p_1=p_1(\e,\bar K_1)$ in such a way that
for every $K\ge \bar K_1$ and $p\in (\bar p_1,p_M)$ we have 
	\begin{equation}\label{pezzo-f_p-piccolo}
	\int_{G_{i,p}(K)} r^{M-3}f_p\psi_{j,p}^2 dr  < \e \int_{G_{i,p}(K)} r^{M-3}\psi_{j,p}^2\underset{\eqref{normalization}}{\le} \e
	\end{equation}
	as $i=1,\dots,m$ and $j=1,\dots,m-1$.
	Hence multiplying equation \eqref{radial-general-H-no-c} for $\psi_{j,p}$ and integrating over $G_{i,p}(K)$ yields
	\begin{align}\nonumber 
	- \nu_j(p) \int_{G_{i,p}(K)} r^{M-3} \psi_{j,p}^2 dr & = \int_{G_{i,p}(K)} (r^{M-1}\psi_{j,p}')'\psi_{j,p}dr +\int_{G_{i,p}(K)} r^{M-3}f_p\psi^2_{j,p} dr \\\label{passo0} \underset{\eqref{pezzo-f_p-piccolo}}{<} & \int_{G_{i,p}(K)} (r^{M-1}\psi_{j,p}')'\psi_{j,p}dr +\e .\end{align}
	Next we write 
	\[\begin{array}{rlll}
	\a=& K(\widetilde {\mathcal M}_{i-1,p})^{-1} & \text{ as $i=1,\dots m$,} & \\
	\beta=& (K\widetilde {\mathcal M}_{i,p})^{-1}& \text{ if $i=1,\dots m-1$,} &  \quad \text{ or }\beta=1 \; \text{ if } i=m, 
	\end{array}\]
	so that $G_{i,p}(K)=(\alpha,\beta)$ and integrating by parts we have
	\begin{align*}
	\int_{G_{i,p}(K)} (r^{M-1}\psi_{j,p}')'\psi_{j,p}dr =  -\int_{G_{i,p}(K)} r^{M-1}(\psi_{j,p}')^2dr + \b^{M-1}\psi_{j,p}'(\beta) \psi_{j,p}(\beta)  \\
	- \a^{M-1}\psi_{j,p}'(\a) \psi_{j,p}(\a).
	\end{align*}
	But by the definition of $\widetilde \psi_{j,p}$ we have
	\begin{align*} 
	\a^{M-1}\psi_{j,p}'(\a) \psi_{j,p}(\a)& =K^{M-1} \widetilde \psi^{i-1}_{j,p}(K)\, (\widetilde \psi^{i-1}_{j,p})'(K), \\
	\beta^{M-1}\psi_{j,p}'(\beta) \psi_{j,p}(\beta)&=K^{1-M} \widetilde \psi^i_{j,p}(K^{-1})\,(\widetilde \psi_{j,p}^i)'(K^{-1}),
	\intertext{if $i=1,\dots, m-1$, or}
	\beta^{M-1}\psi_{j,p}'(\beta) \psi_{j,p}(\beta)&= 0
	\end{align*}
	if $i=m$.
	Therefore the convergence in Proposition \ref{prop:limit-nu-j<m}  implies that when $p\to p_M$
	\begin{align*}
	& \b^{M-1}\psi_{j,p}'(\beta) \psi_{j,p}(\beta)  - \a^{M-1}\psi_{j,p}'(\a) \psi_{j,p}(\a) \\
	& \to
	(A^i_j)^2 K^{1-M} \eta_1(K^{-1})\,\eta_1'(K^{-1}) - (A^{i-1}_j)^2 K^{M-1} \eta_1(K)\,\eta_1'(K) 
	\intertext{ if $i=1,\dots m-1$, or }
	& \to  -(A^{m-1}_j)^2 K^{M-1} \eta_1(K)\,\eta_1'(K)
	\end{align*}
	if $i=m$. 
	Besides there exists $\bar K\ge  \bar K_1$ so that for any $K>\bar K$
	\begin{equation}\label{pezzo-integrato-piccolo}
	-\e < K^{M-1} \eta_1(K)\eta_1'(K) <  K^{1-M}\eta_1(K^{-1})\eta_1'(K^{-1}) < \e .
	\end{equation}
	This choice is possible because $\eta_1$ has only one critical point, which is a maximum,  and $\eta_1(t)$ , $t^{M-1} \eta_1(t)\eta_1'(t) \to 0$ as $t\to 0$ and $t\to \infty$.
	Then  we can chose $p_2=p_2(\e, \bar K)$ in such a way that 
	\begin{align*}\label{eq:passo}
	\int_{G_{i,p}(K)} (r^{M-1}\psi_{j,p}')'\psi_{j,p}dr < & -\int_{G_{i,p}(K)} r^{M-1}(\psi_{j,p}')^2dr  + A \e  \le A \e 
	\end{align*}
	for $p\in (p_2, p_M)$ for any $K>\bar K$. Here the constant $A$ only depends by the coefficients $A^i_j$.
	Inserting this bound into \eqref{passo0} gives
	\begin{align*}
	-\nu_j(p) \int_{G_{i,p}(K)} r^{M-3} \psi_{j,p}^2 dr < (1+A)\e 
	\end{align*}
	in the same range of the parameter $p$.
	Moreover \ref{prop:limit-nu-j<m} yields that also  $-\nu_j(p) > (M-1)(1-\e)$, possibly increasing $p_2$.
	Hence recalling that $\e<1/2$ we get 
	\[ \int_{G_{i,p}(K)} r^{M-3} \psi_{j,p}^2 dr < \frac{1 + A}{M-1} \  \frac{\e}{1-\e} \le  C \e \] 
	where $C$ only depends by $A$ and $M$, and this concludes the proof.
\end{proof}

\begin{lemma}\label{lem:ortog-limite}
	The constants $A_j^i$ in Proposition \ref{prop:limit-nu-j<m} satisfy
	\begin{equation}\label{normalization-2}
	\sum_{i=0}^{m-1}A_j^iA_k^i \int\limits_{0}^{+\infty} r^{M-3} \eta_1^2dr =\delta_{jk}
	\end{equation}
	for every $j,k=1,\dots,m-1$.
\end{lemma}
\begin{proof}
Let 
\[H(p) : = \int_0^1 r^{M-3}\psi_{j,p} \psi_{k,p} dr- \sum_{i=0}^{m-1}A_j^iA_k^i \int\limits_{0}^{+\infty} r^{M-3} \eta_1^2 dr .\]
By \eqref{normalization} we have 
\[ \delta_{jk} - \sum_{i=0}^{m-1}A_j^iA_k^i \int\limits_{0}^{+\infty} r^{M-3} \eta_1^2 dr = H(p)\]
for every $p\in (1,p_M)$, and the claim can be proved by showing that $H(p_n)\to 0$ for the sequence $p_n$ which realizes
\[ \widetilde \psi^i_{j,p_n} \to A^i_j \eta_1 \]
as $i=0,\dots m-1$ and  $j=1,\dots m-1$,	according to Proposition \ref{prop:limit-nu-j<m}.
More precisely we will show that for any $\e>0$ we can chose $\bar n$ in such a way that $|H(p_n)| < \e$ as $n >\bar n$.
To not make notation even heavier, in the following we shall write $p$ meaning $p_n$, and $p\in (\bar p, p_M)$ meaning $n > \bar n$.

Let $K>1$ a parameter to be chosen later on according to $\e$; we split the interval $(0,1)$ in the same way used in Lemma \ref{lemma-2.5} and write
\begin{align*}
H(p) & = \sum_{i=1}^m \int\limits_{G_{i,p}(K)} r^{M-3}\psi_{j,p} \psi_{k,p} dr+\int\limits_0^{K (\widetilde{ \mathcal M}_{0,p})^{-1}} r^{M-3}\psi_{j,p} \psi_{k,p}dr \\ 
& +  \sum_{i=1}^{m-1} \int\limits_{(K \widetilde{ \mathcal M}_{i,p})^{-1}}^{ K (\widetilde{ \mathcal M}_{i,p})^{-1} }r^{M-3}\psi_{j,p} \psi_{k,p} dr   - \sum_{i=0}^{m-1}A_j^iA_k^i \int\limits_{0}^{+\infty} r^{M-3} \eta_1^2 dr .
\end{align*}
Now 
\[
\left|\int\limits_{G_{i,p}(K)} r^{M-3}\psi_{j,p}\psi_{k,p}  dr \right| \le \left(\int\limits_{G_{i,p}(K)} r^{M-3}\psi_{j,p}^2 dr\right)^{\frac{1}{2}}\left(\int\limits_{G_{i,p}(K)} r^{M-3}\psi_{k,p}^2 dr\right)^{\frac{1}{2}}  ,
\]%
so Lemma \ref{lem:int-va-zero} yields that we can chose $\bar K_0=\bar K_0(\e)$ and $\bar p_0=\bar p_0(\e, K_0)$ in such a way that 
\begin{equation}\label{pezzo-Gi-piccolo-0} 
\left|\int\limits_{G_{i,p}(K)} r^{M-3}\psi_{j,p}\psi_{k,p}  dr \right|  <\e/3m
\end{equation} 
for $K\ge  \bar K_0$ and $p\in (\bar p_0, p_M)$.
\\
Besides rescaling and using the convergence in Proposition \ref{prop:limit-nu-j<m}, it is easy to see that for every $K$
\[\begin{split}
&\int\limits_0^{K (\widetilde{ \mathcal M}_{0,p})^{-1}} r^{M-3}\psi_{j,p} \psi_{k,p} dr =\int\limits_{0}^K r^{M-3} \widetilde{\psi}_{j,p}^0\widetilde{\psi}_{k,p}^0 dr\to  A_j^0A_k^0 \int\limits_{0}^K r^{M-3} \eta_1^2dr 
\end{split}\]
as $p\to p_M$, as well as 
\[\begin{split}
& \int\limits_{(K \widetilde{ \mathcal M}_{i,p})^{-1}}^{ K (\widetilde{ \mathcal M}_{i,p})^{-1} }r^{M-3}\psi_{j,p} \psi_{k,p} dr =\int\limits_{K^{-1}}^K r^{M-3} \widetilde{\psi}_{j,p}^i\widetilde{\psi}_{k,p}^i dr\to  A_j^iA_k^i \int\limits_{K^{-1}}^K r^{M-3} \eta_1^2 dr 
\end{split}
\]
for $i=1,\dots,m-1$.
Since $r^{M-3}\eta_1^2\in L^1(0,\infty)$, there exists $\bar K_1=\bar K_1(\e)>1$ such that 
\[
|A^0_j A^0_k| \int_K^{\infty} r^{M-3} \eta_1^2 dr  +  \sum_{i=1}^{m-1} |A^i_j A^i_k| \left( \int\limits_0^{K^{-1}} r^{M-3} \eta_1^2 dr  +
\int\limits_K^\infty r^{M-3} \eta_1^2 dr
\right)< \e/3
\]

as $K>\bar K_1$ and consequently for any $K>\bar K_1$ we can chose $p_1 = p_1(\e, K)$ in such a way that
\begin{equation} \label{altro-pezzo-piccolo}\begin{split}
\left|\int\limits_0^{K (\widetilde{ \mathcal M}_{0,p})^{-1}} r^{M-3}\psi_{j,p} \psi_{k,p}dr+  \sum_{i=1}^{m-1} \int\limits_{(K \widetilde{ \mathcal M}_{i,p})^{-1}}^{ K (\widetilde{ \mathcal M}_{i,p})^{-1} }r^{M-3}\psi_{j,p} \psi_{k,p} dr  \right.\\
\left.- \sum_{i=0}^{m-1}A_j^iA_k^i \int\limits_{0}^{+\infty} r^{M-3} \eta_1^2 dr \right| < 2\e/3
\end{split}	\end{equation}    
for every $p\in (p_1, p_M)$.
 Putting together \eqref{pezzo-Gi-piccolo-0} and \eqref{altro-pezzo-piccolo} gives the claim.
\end{proof}

\begin{corollary}\label{rem-index-k}
	There exists an index $k\in\{0,1,\dots m-1\}$ such that 
	\[\sum_{j=1}^{m-1}(A_j^{k})^2 < \left(\int_0^{\infty} t^{M-3}\eta_1^2 dt \right)^{-1}.\]
\end{corollary}
\begin{proof}
	Let $C= \left(\int_0^{\infty} t^{M-3}\eta_1^2 dt \right)^{-1}$.  Using \eqref{normalization-2} with $j=k$ we immediately have
		\[
		\sum_{i=0}^{m-1}(A_j^{i})^2 =C\]
		for every $j=1,\dots m-1$.
	Therefore 
	\[ \sum_{i=0}^{m-1}\left( \sum_{j=1}^{m-1}(A_j^i)^2\right)  = \sum_{j=1}^{m-1}\left( \sum_{i=0}^{m-1}(A_j^i)^2\right) = (m-1) C .\]
	Since all the $m$ terms $\sum\limits_{j=1}^{m-1}(A_j^i)^2$ are nonnegative, at least one among them should satisfy
	\[ \sum_{j=1}^{m-1}(A_j^i)^2 \le \frac{m-1}{m} C < C .\]
\end{proof}

Such index $k$ will play a role in the proof of next proposition, which is the main result in the present subsection.

\begin{proposition}\label{prop:num-lim} 	
We have 
\[ \lim\limits_{p\to p_M} \nu_m(p)= -(M-1) .\]
Moreover for any sequence $p_n$ in $(1,p_M)$ with $p_n\to p_M$ there exist an extracted sequence (that we still denote by $p_n$), and $m$ numbers $A^0_m, \dots A^{m-1}_m\in \R$ such that
	\[\widetilde{\psi}_{m,p_n}^i\to A^i_m\eta_1\] 
	weakly in ${\mathcal D}_{M}(0,\infty)$ and  strongly  in $C^1_{\loc} (0,\infty)$. 
	\end{proposition} 
		\begin{proof}
By virtue of Lemma \ref{autofunz-limite} and Remark \ref{remark-dim} it is enough to show that 
	\[\lim\limits_{p\to p_M} \nu_m(p)= -(M-1).\] 
	Moreover, thanks to \eqref{num>k-general-H}, it suffices to check that 
\[\limsup\limits_{p\to p_M} \nu_m(p) \le -(M-1).\]
We therefore chose a sequence $p_n\to p_M$ such that $\nu_m(p_n)\to \limsup\limits_{p\to p_M} \nu_m(p)$. Possibly passing to an extracted sequence, we may assume w.l.g.~that $ \widetilde \psi^i_{j,p_n}\to A^i_j \eta_1$ as $i=0,\dots m-1$ and $j=1,\dots m-1$, in force of Proposition \ref{prop:limit-nu-j<m}. To not make notation even heavier, in the following we shall write $p$, meaning $p_n$.

Now the claim follows by producing, for every $\e>0$, a family of nontrivial test functions $\psi_p  \in H^1_{0,M}$,  $\psi_p \underline{\perp}_{M} \{\psi_{1,p},\dots,\psi_{m-1,p}\}$, such that
\begin{equation}\label{test-lim} \begin{split}
\limsup\limits_{p\to p_M} {\mathcal R}_p(\psi_p) \le -(M-1) +\e  , \\ {\mathcal R}_p(\psi) :=
\frac{ \int_0^1r^{M-1}(\psi')^2-r^{M-3}f_p(r) \psi^2 dr }{\int_0^1 r^{M-3}\psi^2 dr },
\end{split}\end{equation}
and recalling  the variational characterization \eqref{nu-var}. 

Let us consider the index $k$ in Corollary \ref{rem-index-k} and define 
\[\psi_p(r):= (\eta_1\Phi)\big( r \widetilde{ \mathcal M}_{k,p} \big) +\sum_{j=1}^{m-1} a_{j,p} \psi_{j,p}(r) ,\]
where $\Phi\in C^{\infty}_0(0,\infty)$ is a cut-off function  with   
\begin{align}
\label{Phi1}
0\le \Phi(r)\le 1 , \; \mbox{ for every  } r\in[0,\infty) , \\
\label{Phi2}
\Phi(r)  = \begin{cases} 0 & \mbox{ if $r\in [0,(2R)^{-1}]$ or $[2R,\infty)$}, \\
1 & \mbox{ if } r\in [R^{-1}, R] ,
\end{cases}
\\
\label{Phi3}
|\Phi'(r)|  \le   \begin{cases}  2 R & \mbox{ if } r\in [(2R)^{-1} , R^{-1}], \\
2 R^{-1} & \mbox{ if } r\in [R , 2R] 
\end{cases}
\end{align} 
and $\eta_1$ as defined in \eqref{eta}.
Here $R$ is a parameter to be suitably chosen, depending on $\e$.
Since we will send $p\to p_{M}$, thanks to  \eqref{hp-a} and \eqref{hp-c} we may take w.l.g.~that 
	\begin{equation}\label{rosso}
t_{k,p} \widetilde{ \mathcal M}_{k,p} < (2R)^{-1} < 2R < 	t_{k+1,p} \widetilde{ \mathcal M}_{k,p} \le   \widetilde{ \mathcal M}_{k,p}.
	\end{equation} 
The coefficients $a_{j,p}$, instead, are chosen in such a way to ensure that $\psi_p\underline{\perp}_{ M}\{\psi_{1,p},\dots,\psi_{m-1,p}\}$ for every $p$, namely
\begin{align*}
a_{j,p} & =- \int_0^1 r^{M-3} \psi_{j,p}(r) (\eta_1\Phi)\big( r \widetilde{ \mathcal M}_{k,p} \big)dr .
\intertext{By \eqref{Phi2} and \eqref{rosso} we have}
a_{j,p} & =-{\int_{t_{k,p}}^{t_{k+1,p}} }r^{M-3} \psi_{j,p}(r) (\eta_1\Phi)\big( r \widetilde{ \mathcal M}_{k,p} \big)dr ,
\intertext{so performing the change of variables $t= r \widetilde{ \mathcal M}_{k,p} $  and recalling the definition of $\widetilde\psi_{j,p}^k$ in \eqref{rescaled-eigenf} one gets} 
a_{j,p}& = - (\widetilde{ \mathcal M}_{k,p})^{-\frac{M-2}{2}} \int _{t_{k,p}\widetilde{ \mathcal M}_{k,p} }^{ t_{k+1,p}\widetilde{ \mathcal M}_{k,p}}  t^{M-3}\widetilde\psi_{j,p}^k\eta_1\Phi dt = (\widetilde{ \mathcal M}_{k,p})^{-\frac{M-2}{2}} \widetilde a_{j,p}
\intertext{ for }
\widetilde a_{j,p} & = - \int_0^{+\infty} t^{M-3}\widetilde\psi_{j,p}^k\eta_1\Phi dt
\end{align*}

Obtaining \eqref{test-lim} will request many computations, that we split in several claims.

Claim 1:
\begin{equation}\label{claim1}  \begin{split}
{\mathcal D}(p) & := \int_0^1 r^{M-3}\psi_p^2 dr \\
& =  (\widetilde{ \mathcal M}_{k,p})^{2-M} \left[  \int_0^{\infty} t^{M-3}\left(\eta_1\Phi\right)^2 (t) dt - \sum\limits_{j=1}^{m-1} (\widetilde a_{j,p})^2 \right].\end{split}\end{equation}

It suffices to compute
\begin{align*}
{\mathcal D}(p) & = \int_0^1 r^{M-3}\left(\eta_1\Phi\right)^2 \big( r \widetilde{ \mathcal M}_{k,p}\big)  dr
+ \sum\limits_{j,k=1}^{m-1} a_{j,p}a_{k,p} \int_0^1 r^{M-3}\psi_{j,p}\psi_{k,p} dr \\
& + 2 \sum\limits_{j=1}^{m-1} a_{j,p}\int_0^1 r^{M-3}\psi_{j,p} \left(\eta_1\Phi\right)\big( r \widetilde{ \mathcal M}_{k,p}\big) dr ,
\end{align*}
where performing the change of variables $t= r \widetilde{ \mathcal M}_{k,p}$ in the first integral and taking advantage from \eqref{Phi2} and \eqref{rosso}  
we have
\begin{align*} \int_0^1 r^{M-3}\left(\eta_1\Phi\right)^2 \big( r \widetilde{ \mathcal M}_{k,p}\big)  dr= & (\widetilde{ \mathcal M}_{k,p})^{2-M} {\int_0^{\widetilde{ \mathcal M}_{k,p}}} t^{M-3}\left(\eta_1\Phi\right)^2 (t) dt
\\
= &  (\widetilde{ \mathcal M}_{k,p})^{2-M} \int_0^{\infty} t^{M-3}\left(\eta_1\Phi\right)^2 (t) dt  .
\end{align*}
Next using \eqref{normalization} and the definition of $a_{j,p}$, $\widetilde a_{j,p}$ in the second and third integrals gives \eqref{claim1}.
Further Claim 2:
\begin{equation}\label{claim2}
\begin{split}{\mathcal N}_1(p)  & :=\int_0^1 r^{M-3} f_p\psi_p^2 dr = (\widetilde{ \mathcal M}_{k,p})^{2-M} \int_0^{\infty} t^{M-3} \widetilde f_{k,p}(t) \left(\eta_1\Phi\right)^2 (t) dt\\
&  + \sum\limits_{j,k=1}^{m-1} a_{j,p}a_{k,p} \int_0^1 r^{M-3} f_p\psi_{j,p}\psi_{k,p} dr   \\
& 
 + 2 \sum\limits_{j=1}^{m-1} a_{j,p}\int_0^1 r^{M-3} f_p(r) \psi_{j,p}(r) \left(\eta_1\Phi\right)\big( r \widetilde{ \mathcal M}_{k,p}\big) dr 
 \end{split}
 \end{equation}
where $\widetilde f_{k,p}$ is as defined in \eqref{tilde-f_p}. 
Indeed it suffices to write explicitly 
\begin{align*}
{\mathcal N}_1(p) & = \int_0^1 r^{M-3}f_p(r) \left(\eta_1\Phi\right)^2 \big( r \widetilde{ \mathcal M}_{k,p}\big)  dr
+ \sum\limits_{j,k=1}^{m-1} a_{j,p}a_{k,p} \int_0^1 r^{M-3} f_p\psi_{j,p}\psi_{k,p} dr \\
& + 2 \sum\limits_{j=1}^{m-1} a_{j,p}\int_0^1 r^{M-3} f_p (r) \psi_{j,p}(r) \left(\eta_1\Phi\right)\big( r \widetilde{ \mathcal M}_{k,p}\big) dr ,
\end{align*}
perform the change of variables $t= r \widetilde{ \mathcal M}_{k,p}$  and taking again advantage from \eqref{Phi2} and \eqref{rosso} in the first integral.

Besides, Claim 3:
\begin{equation}
\label{claim3}\begin{split}
{\mathcal N}_2(p)  & :=\int_0^1 r^{M-1}(\psi')^2 dr = (\widetilde{ \mathcal M}_{k,p})^{2-M} \left[
-(M-1) \int_0^{\infty} t^{M-3} (\eta_1\Phi)^2 dt \right.\\
&  \left. + \int_0^{\infty} t^{M-1}  W (\eta_1\Phi)^2 dt  + \int_0^{\infty} t^{M-1} (\eta_1\Phi')^2 dt  - \sum\limits_{j=1}^{m-1} \nu_j(p) (\widetilde a_{j,p})^2  \right] \\
& +  \sum\limits_{j,k=1}^{m-1} a_{j,p}a_{k,p} \int_0^1 r^{M-3} f_p \psi_{j,p} \psi_{k,p} dr \\
& + 2 \sum\limits_{j=1}^{m-1} a_{j,p}\int_0^1 r^{M-3} f_p \psi_{j,p} \, (\eta_1\Phi)\big( r \widetilde{ \mathcal M}_{k,p}\big) dr .
\end{split}\end{equation}

By definition 
\begin{align*}
{\mathcal N}_2(p) & = \int_0^1 r^{M-1}\left(\left((\eta_1\Phi)\big( r \widetilde{ \mathcal M}_{k,p}\big)\right)' \right)^2  dr
+ \sum\limits_{j,k=1}^{m-1} a_{j,p}a_{k,p} \int_0^1 r^{M-1}\psi'_{j,p}\psi'_{k,p} dr \\
& + 2 \sum\limits_{j=1}^{m-1} a_{j,p}\int_0^1 r^{M-1}\psi'_{j,p} \left((\eta_1\Phi)\big( r \widetilde{ \mathcal M}_{k,p}\big)\right)' dr 
\end{align*}
As for the first term, we have 
\begin{align*}
\int_0^1 r^{M-1}\left(\left((\eta_1\Phi)\big( r \widetilde{ \mathcal M}_{k,p}\big)\right)' \right)^2  dr  = (\widetilde{ \mathcal M}_{k,p})^2 \int_0^1 r^{M-1}\left((\eta_1\Phi)'\big( r \widetilde{ \mathcal M}_{k,p}\big) \right)^2  dr \\
 =  (\widetilde{ \mathcal M}_{k,p})^{2-M} \int_0^{\infty} t^{M-1} \left((\eta_1\Phi)'\right)^2 dt
\intertext{ after performing the change of variables $t= r \widetilde{ \mathcal M}_{k,p}$ and recalling \eqref{Phi2},  \eqref{rosso}. \newline
Next we decompose $\left((\eta_1\Phi)'\right)^2 = \eta_1' \left(\eta_1\Phi^2\right)' + (\eta_1 \Phi')^2$, so that}
 =   (\widetilde{ \mathcal M}_{k,p})^{2-M} \left( \int_0^{\infty} t^{M-1}  \eta_1' \left(\eta_1\Phi^2\right)' dt + \int_0^{\infty} t^{M-1} (\eta_1\Phi')^2 dt \right)
\intertext{and remembering that $\eta_1$ is the first eigenfunction for \eqref{eq:finale} and solves \eqref{eq:finale-weak} with $\beta_1=-(M-1)$, we have}
 = (\widetilde{ \mathcal M}_{k,p})^{2-M} \left( -(M-1) \int_0^{\infty} t^{M-3} (\eta_1\Phi)^2 dt + \int_0^{\infty} t^{M-1}  W (\eta_1\Phi)^2 dt  \right. \\
 \left.  + \int_0^{\infty} t^{M-1} (\eta_1\Phi')^2 dt \right).
\end{align*}
Next  \eqref{radial-general-weak-H-no-c} yields
\begin{align*}
\int_0^1 r^{M-1}\psi'_{j,p}\psi'_{k,p} dr & =
\int_0^1 r^{M-3} f_p \psi_{j,p}\psi_{k,p} dr + \nu_j(p)\delta_{jk} 
\end{align*}
thanks to \eqref{normalization}.
Concerning the last term, equation   \eqref{radial-general-weak-H-no-c} again gives
\begin{align*}
\int_0^1 r^{M-1}\psi'_{j,p} \left((\eta_1\Phi)\big( r \widetilde{ \mathcal M}_{k,p}\big)\right)' dr 
& =
\int_0^1 r^{M-3} f_p \psi_{j,p}(r)(\eta_1\Phi)\big( r \widetilde{ \mathcal M}_{k,p}\big) dr \\
& + \nu_j(p) \int_0^1 r^{M-3} \psi_{j,p}(r) \, (\eta_1\Phi)\big( r \widetilde{ \mathcal M}_{k,p}\big) dr \\
& = \int_0^1 r^{M-3} f_p \psi_{j,p}(r)(\eta_1\Phi)\big( r \widetilde{ \mathcal M}_{k,p}\big) dr - \nu_j(p) a_{j,p}
\end{align*}
So the claim follows after summing up the three terms.

Summing up, \eqref{claim1}, \eqref{claim2} and \eqref{claim3} give 
\begin{align*}
{\mathcal R}_p(\psi_p) &  = \frac{{\mathcal N}_2(p) -{\mathcal N}_1(p) }{{\mathcal D}(p)} = -(M-1) +\frac{{\mathcal A}_p(\Phi)}{{\mathcal B}_p(\Phi)} 
\intertext{where }
{\mathcal A}_p(\Phi)& = \int_0^{\infty} t^{M-3}(t^2W-\widetilde f_{k,p})(\eta_1\Phi)^2 dt + \int_0^{\infty} t^{M-1}(\eta\Phi')^2 dt \\
& -\sum\limits_{j=1}^{m-1} (\nu_j(p)+M-1) (\widetilde a_{j,p})^2  \\
{\mathcal B}_p(\Phi)& =	\int_0^{\infty} t^{M-3} (\eta_1\Phi)^2 dt - \sum\limits_{j=1}^{m-1} (\widetilde a_{j,p})^2
\end{align*}
But when  $p\to p_M$, then $\widetilde f_{k,p} \to F = t^2W$ uniformly on $[R^{-1}, R]$ by Lemma \ref{lem:tilde-f_p}, so that 
\[ \int_0^{\infty} t^{M-3} \left(t^2W- \widetilde f_{k,p}\right) \left(\eta_1\Phi\right)^2 dt \to 0 . \]
Besides Proposition \ref{prop:limit-nu-j<m} assures that $\nu_j(p)+M-1 \to 0$ and that
\begin{align*}
\widetilde a_{j,p} & = - \int_0^{+\infty} t^{M-3}\widetilde\psi_{j,p}^{k}\eta_1\Phi dt  \to - A^{k}_j \int_0^{+\infty} t^{M-3}\eta_1^2\Phi dt 
\end{align*}
as $p\to p_M$. 
Therefore
\begin{align*}
\lim\limits_{p\to p_M} {\mathcal R}_p(\psi_p) & = -(M-1) \\
& + \dfrac{\int_0^{\infty} t^{M-1}(\eta_1\Phi')^2 dt }{\int_0^{\infty} t^{M-3} (\eta_1\Phi)^2 dt - \left( \int_0^{+\infty} t^{M-3}\eta_1^2\Phi dt \right)^2\sum\limits_{j=1}^{m-1} (A^{k}_j)^2 } \end{align*}

We conclude the proof by showing that for every $\e>0$ it is possible to chose $R$ and the  cut-off function $\Phi$ satisfying \eqref{Phi1}--\eqref{Phi3} in such a way that  
\[ \dfrac{\int_0^{\infty} t^{M-1}(\eta_1\Phi')^2 dt }{\int_0^{\infty} t^{M-3} (\eta_1\Phi)^2 dt - \left( \int_0^{+\infty} t^{M-3}\eta_1^2\Phi dt \right)^2\sum\limits_{j=1}^{m-1} (A^{k}_j)^2 } < \e .\]

To begin with 
\begin{align*}
\int_0^{\infty} t^{M-1}(\eta_1\Phi')^2 dt & = \int_{\frac{1}{2R}}^{\frac{1}{R}}t^{M-1}(\eta_1\Phi')^2 dt + \int_{R}^{2R}t^{M-1}(\eta\Phi')^2 dt 
\\
& \underset{\eqref{Phi3}}{\le } C R^2 \int_{\frac{1}{2R}}^{\frac{1}{R}}t^{M-1}\eta_1^2 dt + \frac{C}{R^2}\int_{R}^{2R}t^{M-1}\eta_1^2 dt
\intertext{and since $\eta_1$ has a unique maximum point in $\bar t\in (0,+\infty)$, if $R >\max\{\bar t, 1/\bar t\}$ we have} 
& \le  C R^2 \left(\eta_1\left(\frac{1}{R}\right)\right)^2\int_{\frac{1}{2R}}^{\frac{1}{R}}t^{M-1}dt + \frac{C}{R^2}\left(\eta_1(2R)\right)^2\int_{R}^{2R}t^{M-1} dt
\\
& =  \frac{C^2(1-2^{-M})}{M R^M\left(1+\frac{1}{M(M-2)R^2}\right)^{M}} + \frac{C^2(2^{M}-1) R^M}{M\left(1+\frac{R^2}{M(M-2)}\right)^{M}} =o(1)
\end{align*}
as $R\to \infty$.
Next it is clear that 
\[ \int_0^{\infty} t^{M-3} (\eta_1\Phi)^2 dt \to \int_0^{\infty} t^{M-3} \eta_1^2 dt > 0 \]
as $R\to \infty$, because 
\begin{align*}
0 \underset{\eqref{Phi1}}{\le} &  \int_0^{\infty} t^{M-3} \eta_1^2 dt - \int_0^{\infty} t^{M-3} (\eta_1\Phi)^2 dt \\
\underset{\eqref{Phi2}}{=}  & 
\int_0^{\frac{1}{R}} t^{M-3} \eta_1^2 (1-\Phi^2) dt + \int_{R}^{\infty} t^{M-3} \eta_1^2 (1-\Phi^2) dt \\
\underset{\eqref{Phi1}}{\le} &  \int_0^{\frac{1}{R}} t^{M-3} \eta_1^2 dt + \int_{R}^{\infty} t^{M-3} \eta_1^2 dt = o(1)
\end{align*}
 since $\int_0^{\infty} t^{M-3} \eta_1^2 dt<\infty$.
Similarly
\[ \int_0^{\infty} t^{M-3} \eta_1^2\Phi dt \to \int_0^{\infty} t^{M-3} \eta_1^2 dt > 0 .\]
Eventually
\[ \begin{split} \int_0^{\infty} t^{M-3} (\eta_1\Phi)^2 dt - \left( \int_0^{+\infty} t^{M-3}\eta_1^2\Phi dt \right)^2\sum\limits_{j=1}^{m-1} (A^{k}_j)^2 
 \\ \longrightarrow
\int_0^{\infty} t^{M-3} \eta_1^2 dt -  \left(\int_0^{\infty} t^{M-3} \eta_1^2 dt)\right)^2 \sum\limits_{j=1}^{m-1} (A^{k}_j)^2 \neq 0\end{split}\]
by Corollary \ref{rem-index-k}, which ends the proof.
\end{proof}

We are now in position to prove Theorem \ref{teo:morse-index}.
\begin{proof}
Propositions \ref{prop:limit-nu-j<m} and \ref{prop:num-lim} prove that each generalized radial singular negative eigenvalue $\widehat \nu_i(p)\to -(M-1)$ as $p\to p_M$ for $i=1,\dots,m$. 
Inserting these asymptotic values into \eqref{tag-2-H} gives that $J_i(p)\to 1+\frac {\a}2$ as $p\to p_\a=p_M$ for $j=1,\dots,m$. In particular from \eqref{nl<k-general-H} and \eqref{num>k-general-H} we have $J_i(p)\nearrow 1+\frac {\a}2$ for $j=1,\dots,m-1$ while $J_m(p)\searrow 1+\frac {\a}2$.
Then, when $\a$ is not an even integer all the eigenvalues $\widehat \nu_i(p)$ gives the same contribution to the Morse index giving \eqref{morse-index-p-alpha}. When $\a$ is an even integer instead in the sum in \eqref{tag-2-H} we have to add the contribution of all the $m$ eigenvalues for $j\leq \frac \a 2$ and the contribution of only $ m-1$ eigenvalues for $j=1+\frac \a 2$, which gives \eqref{morse-index-p-alpha-even}.
\end{proof}

\section{Nondegeneracy and small perturbations}\label{se:4}

In this section we address to the nondegeneracy of radial solutions to \eqref{H} when $p$ approaches $p_\a$ and we prove 
Theorem \ref{teo:nondegeneracy} and its consequence Theorem \ref{teo:esistenza}. We recall that a solution $u$ to \eqref{H} is said nondegenerate if the linearized operator at $u$, $L_u$, does not admit zero as an eigenvalue in $H^1_0(B)$, and hence if the linearized equation at $u$, namely
\begin{equation} \label{eq:linearized}
\left\{\begin{array}{ll}
-\Delta \psi = p|x|^{\alpha}|u|^{p-1} \psi \qquad & \text{ in } B, \\
\psi= 0 & \text{ on } \partial B,
\end{array} \right.
\end{equation}
does not admit any nontrivial solution in $H^1_0(B)$. Degeneracy can be computed by analyzing the singular Sturm-Liouville eigenvalue problem related to the transformed function $v_p$ introduced in \eqref{transformation-henon-no-c} as in the previous section. Indeed 
degeneracy of radial solutions to \eqref{H} has been characterized in \cite{AG-sez2} using the singular negative radial eigenvalues $\widehat \nu_k(p)$, defined in \eqref{nu-var}, for $k=1,\dots,m$.  Putting together Theorems 1.6 and 1.7 in \cite{AG-sez2} we obtain
\begin{proposition}\label{prop:charact-nondegeneracy}
Let $\a\geq 0$ and $p\in (1,p_\a)$. 
A radial solution $u_p$ to \eqref{H} with $m$ nodal zones is radially nondegenerate and it is degenerate if and only 
\[
\widehat \nu_k(p)=-\left(\frac2{2+\a}\right)^2 j(N-2+j)
\]
for some $k=1,\dots,m$ and for some $j\geq 1$.
\end{proposition}
Therefore the asymptotic nondegeneracy of $u_p$ as $p\to p_\a$ can be deduced, via the transformation \eqref{transformation-henon-no-c}, by the asymptotic behavior of the radial singular eigenvalues $\widehat \nu_k(p)$ as $p\to p_M$. Indeed by the analysis performed in Section \ref{se:3} we have: 

\begin{proof}[Proof of Theorem \ref{teo:nondegeneracy}]
	
	Let us denote by $g(s)$ the decreasing function
	\[ g(s):= -s(N-2+s) , \qquad s\ge 0 . \]
	By Proposition \ref{prop:charact-nondegeneracy} $u_p$ is degenerate if and only if there is some $k=1,\dots, m$ such that 
	\begin{equation}\label{cond-deg} 
	\left(\frac{2+\a}{2}\right)^2 \widehat \nu_k(p)= g(j) \quad \mbox{ for some positive integer $j$.} 
	\end{equation}
	Recalling that $-(M-1)= -\frac{2}{2+\alpha} \left(N-2+\frac{2+\a}{2}\right)$ according to \eqref{M-a},  Propositions \ref{prop:limit-nu-j<m} and \ref{prop:num-lim} imply that
	\begin{equation}\label{cond-deg-lim}  \left(\frac{2+\a}{2}\right)^2 \widehat \nu_k(p) \to g\left( \frac{2+\a}{2}\right) \quad \text{ 	for every  } k=1,\dots m \end{equation}
	as $p\to p_M$.
	Therefore if $\alpha$ is not  a nonnegative even integer, it is easily seen that 
	\[ \left(\frac{2+\a}{2}\right)^2 \widehat \nu_k(p) \in \left( g\left( 2+ \left[\frac{\a}{2}\right]\right) \ , \   g\left( 1+ \left[\frac{\a}{2}\right]\right) \right) \quad \text{ 	for every  } k=1,\dots m \]
	in a left neighborhood of $p_M$, which ensures that \eqref{cond-deg} can not hold since $g$ is strictly decreasing.
	\\
	Otherwise when $\alpha = 2(j-1)$, then \eqref{cond-deg-lim} says that $\left(\frac{2+\a}{2}\right)^2 \widehat \nu_k(p)\to  g(j)$, but  \eqref{nl<k-general-H} and \eqref{num>k-general-H} imply that
	\begin{align*}
	 \left(\frac{2+\a}{2}\right)^2 \widehat \nu_k(p) < g(j)  & \quad \text{ as } k=1,\dots m-1, \\
	\left(\frac{2+\a}{2}\right)^2 \widehat \nu_m(p) > g(j), &
	\end{align*}	
	for every $p\in (1,p_M)$.
	Therefore 
	\begin{align*}
	 \left(\frac{2+\a}{2}\right)^2 \widehat \nu_k(p) \in \left( g(j+1)  , g(j) \right)& \quad \text{ as } k=1,\dots m-1, \\
	\left(\frac{2+\a}{2}\right)^2 \widehat \nu_m(p) \in \left(g(j), g(j-1)\right)&
	\end{align*}	
	in a left neighborhood of $p_M$, and the conclusion follows by the monotonicity of $g$, again.

\end{proof}

As said before the nondegeneracy of $u_p$ has important applications. 
Among them, we mention a procedure  introduced by Davila and Dupaigne in \cite{DD}  which allows to deduce existence results in domains which are perturbations of the ball. We quote also \cite{C} and \cite{AGG} for applications to the H\'enon problem and to nodal solutions annular domains, respectively.  	\\
Let $\sigma : \bar B \to \R^N$ be a smooth function and let 
\[ \Omega_t:=\{x+t\sigma(x) : x\in B\}. \] 
 We want to find solutions to
\begin{equation}\label{Omega-t}
\left\{\begin{array}{ll}
-\Delta u = |x|^{\alpha}|u|^{p-1} u \qquad & \text{ in } \Omega_t, \\
u= 0 & \text{ on } \partial \Omega_t,
\end{array} \right.
\end{equation}

For small values of $t$, the set  $\Omega_t$ is diffeomorphic to $B$ and hence there exists $\tilde \sigma:\bar \Omega_t\to \R^N$ such that $x=y+t\tilde \sigma(y)$ for every $x\in B$ and every $y\in \Omega_t$.  It was noticed in \cite{C} that   if $u(y)$ is a classical solution to \eqref{Omega-t} then $w(x)=u(y)$ is a classical solution to 
\begin{equation} \label{eq:trasformato-in-omega}
\left\{\begin{array}{ll}
-\Delta w - L_t(w) =  |x+t \sigma(x)|^{\alpha}|w|^{p-1} w \qquad & \text{ in } B, \\
w= 0 & \text{ on } \partial B,
\end{array} \right.
\end{equation}
where $L_t$ is the linear operator
\[L_t(w):=t\sum_{i,k}\partial^2_{y_iy_i}\tilde \sigma_k    \partial _{x_k} w+2t \sum_{i,k}\partial_{y_i}\tilde \sigma_k    \partial ^2_{x_ix_k} w+t^2 \sum_{i,j,k}\partial_{y_j}\tilde \sigma_i \partial_{y_j}\tilde \sigma_k \partial ^2_{x_ix_k}w\]
and $\tilde\sigma_k$ denotes the $k$-th component of $\tilde\sigma$.  Observe that $u_p$ solves \eqref{eq:trasformato-in-omega} for $t=0$.
\\
By the nondegeneracy of $u_p$ stated in Theorem \ref{teo:nondegeneracy} it is not hard to deduce the existence of nodal solutions in domains of type $\Omega_t$, i.e. to prove our last result.

\begin{proof}[Proof of Theorem \ref{teo:esistenza}]
When $\a=0$ or $\a>1$ the map 
\[  F:\R\times C^{2,\gamma}_0(\bar B) \to C^{0,\gamma}_0(\bar B) \qquad F(t,w) = -\Delta w - L_t w -  |x+t\sigma|^{\a}|w|^{p-1 }w \]
 where $C^{2,\gamma}_0(\bar B):=\{w\in C^{2,\gamma}(\bar B) : w_{|\partial B}=0\}$,
 is of class $C^1$  for $\gamma$ small enough, and clearly  $F(0,u_p)=0$, where $u_p$ is the radial solution to \eqref{H}.
 Moreover $D_wF(0,u_p)$ (the Fr\'echet derivative of $F$ with respect to $w\in C^{2,\gamma}_0(\bar B)$ computed at $(0,u_p)$) is nothing else than the linearized operator $L_{u_p}$, which is  invertible for $p> \bar p$ appearing in the statement of Theorem \ref{teo:nondegeneracy}, because its  kernel is made up by the solutions of the linearized problem \eqref{eq:linearized}.
 So the  Implicit Function Theorem applies giving a continuum  of functions $w_t\in C^{2,\gamma}_0(\bar B)$  such that $F(t, w_t)=0$. In particular $u_t(y):=w_t(x)$ is a solution of \eqref{eq:trasformato-in-omega}, it has  exactly $m$  nodal zones and its nodal curves does not intersect the boundary, at least for small $t$, thanks to the continuity of the maps $t\mapsto w_t \in C^{2,\gamma}_0(\bar B)$ and $x\to x+t\sigma (x)$.
 
\end{proof}

\section{Appendix}
In the paper \cite{Gidas} Gidas studied with a phase plane analysis the problem
\[\begin{cases}
-u''-\frac{N-1} ru'=u^{\frac{N+2}{N-2}} &  \text{ in }(0,\infty)\\
u>0
\end{cases}\]
and proved that, for $N>2$, the solutions can have the following shapes: 
\begin{align*}
a) \quad & u(r)=\left(\frac {\l\sqrt{N(N-2)}} {\l^2+r^2}\right)^{\frac {N-2}2}, \\
\intertext{where $\l$ is a positive parameter, or }
b) \quad  & u(r)=\left(\dfrac {N-2}2\right)^{\frac {N-2}2}r^{-\frac {N-2}2} ,\\
c)\quad & c_1r^{-\frac {N-2}2}\leq u(r)\leq c_2r^{-\frac {N-2}2}.
\end{align*}
 When $N$ is an integer it has later been proved that only case $a)$ and $b)$ can occur.
This analysis does not need $N$ to be an integer and indeed shows that the unique solutions to problem 
\[ \begin{cases}
-(t^{M-1}V')'=t^{M-1}V^{p_M} 
& \text{ in }t>0\\
V>0
\end{cases}\tag{\ref{equazione-limite-v}}\]
for $M>2$
are the ones in $a)$, $b)$ and $c)$ with $N$ substituted by $M$. In particular the solutions in $a)$ are the unique bounded solutions to \eqref{equazione-limite-v} for every $\l>0$. Imposing also the condition
\[ V(0)=1 \tag{\ref{eq:cond-V-in-zero}} \]
implies that $\l=\sqrt{M(M-2)}$ so that 
\[V_M(r)=\left(1+\frac{r^2}{M(M-2)} \right)^{-\frac{M-2}{2}}\]
as in \eqref{soluzione-limite-v}, is the unique bounded solution to \eqref{equazione-limite-v} that satisfies \eqref{eq:cond-V-in-zero}.

Further we observe that, due the singular behavior at the origin, the solutions $b)$ and $c)$ do not belong to the space $\mathcal{D}_M(0,\infty)$ which is embedded in $L_M^{p_M+1}(0,\infty)$ for $p_M=\frac{M+2}{M-2}$. Therefore the solutions  in $a)$, for every $\l>0$, are also  the only  solutions to \eqref{equazione-limite-v} belonging to $\mathcal{D}_M(0,\infty)$. In particular one sees that  every  solution in $\mathcal{D}_M(0,\infty)$ also belong to $C[0,\infty)$. 

Thus we can also impose the condition \eqref{eq:cond-V-in-zero} obtaining that $V_M$ is the unique $\mathcal{D}_M(0,\infty)$ solution to \eqref{equazione-limite-v} that satisfies \eqref{eq:cond-V-in-zero}. 

\

The previous discussion applies to the study of radial solutions to 
\begin{equation}\label{eq:U-general}
\begin{cases}
-\Delta U=|x|^\a U^{p_\a} & \text{ in }\R^N\\
U>0
\end{cases}
\end{equation}
where $p_\a=\frac{N+2+2\a}{N-2}$. Indeed, it has been proved in \cite{GGN} that
the transformation 
\[t=r^{\frac {2+\a}2}\]
transforms radial $D^{1,2}(\R^N)$ solutions to \eqref{eq:U-general} into $\mathcal {D}_M(0,\infty)$ solutions to \eqref{equazione-limite-v}
with $M$ as in \eqref{M-a} and $M>2$. Performing the previous change of variable into $V_M$ and recalling that $p_\a=p_M$ we get that the unique bounded solutions to \eqref{equazione-limite-v} are given by
\[U_{\a,\l}(x):=\left(\frac {\l \sqrt{(N+\a)(N-2)}}{\l^2+|x|^{2+\a}}  \right)^{\frac {N-2}{2+\a}}\]
and, imposing the condition
\begin{equation}\label{eq:condition-0-U}
U(0)=1
\end{equation}
we get that the unique radial  bounded solution to \eqref{eq:U-general} that satisfies \eqref{eq:condition-0-U}, i.e. the unique solution to \eqref{eq-U},  is 
\[
U_{\a}(x) :=  \left(1+\frac{|x|^{2+\a}}{(N+\alpha)(N-2)}\right)^{-\frac{N-2}{2+\a}} 
\]
as in \eqref{U}. Finally the relation between $D^{1,2}(\R^N)$ and $\mathcal {D}_M(0,\infty)$ also implies that $U_\a$ is the unique $D^{1,2}(\R^N)$
solution to  \eqref{eq:U-general} that satisfies \eqref{eq:condition-0-U}.

Next we look at the generalized radial singular eigenvalue problem associated with the solution $V_M$, namely
\[
-(t^{M-1}\eta')'=t^{M-1}\left(W+\frac \beta{r^2}\right)\eta  \qquad \text{ in }t>0,\\
\tag{\ref{eq:finale}}\]
where $W=\frac{M+2}{M-2}\left(1+\frac{r^2}{M(M-2)} \right)^{-2}$ has been introduced in \eqref{uniform-convergence}, and we look for solutions in $\mathcal D_M(0,\infty)$, namely solutions that satisfy
\[\int_0^{\infty} t^{M-1}\eta'\varphi'\ dt=\int_0^{\infty} t^{M-1}\left(W+\frac \beta{r^2}\right)\eta\varphi \]
for every $\varphi\in C^{\infty}_0(0,+\infty)$. 
 \\
The generalized radial singular eigenvalue problem \eqref{eq:finale} is of the same type of the previous one \eqref{radial-general-H-no-c} and indeed the eigenvalues are defined as far as $\beta<\left(\frac{M-2}2\right)^2$ and they share the same properties of the previous eigenvalues $\widehat \nu(p)$. In particular each eigenvalue is simple and the $i$-th eigenfunction admits $i$ nodal zones. Then we easily seen that 
$\beta_1=-(M-1)$ and $\beta_2=0$ with corresponding eigenfunctions 
\[
\eta_1(r)=\frac r{\big(1+\frac {r^2}{M(M-2)}\big)^{\frac M2}} , \qquad 
\eta_2(r)= \frac{1-\frac{r^2}{M(M-2)}}{\big( 1+\frac {r^2}{M(M-2)}\big)^{\frac M2}}.
\tag{\ref{eta}}\]
The fact that $\beta_2$ is simple implies that $\beta_3>0$, so that $\beta_1$ and $\beta_2$ are the unique non positive eigenvalues of \eqref{eq:finale}. See also \cite{GGN}, where the same properties have been used in the proof of Theorem 1.3.

\end{document}